\documentclass[10pt]{article}
\usepackage{amssymb,color,calrsfs}

\usepackage{amsmath,graphicx,enumerate,textcomp}
\usepackage[abbrev]{amsrefs}
\usepackage[a4paper,top=1.in, bottom=1.5in, left=.8in, right=.8in, headsep=0.in, centering]{geometry}
\usepackage{amssymb,amsfonts,amsthm}
\usepackage{epsfig}
\numberwithin{equation}{section}
\usepackage{lipsum}
\usepackage{graphicx}
\usepackage{esint}

\newcommand{\eps}{\varepsilon}

\newcommand{\grad}{\nabla}

\newtheorem{theorem}{Theorem}[section]
\newtheorem{lemma}{Lemma}[section]
\newtheorem{remark}{Remark}[section]

\newtheorem{cor}{Corollary}[section]
\newtheorem{definition}{Definition}[section]

\newcommand{\R}{{\mathbb R}}

\title{A basic homogenization problem for the $p$-Laplacian in $\R^d$ perforated along a sphere: $L^\infty$ estimates}
\author {Peter  V. Gordon
\thanks{Department of Mathematical Sciences,
Kent State University,
 Kent, OH 44242, USA. E-mail: {\tt gordon@math.kent.edu}}
\and Fedor Nazarov
\thanks{Department of Mathematical Sciences,
Kent State University,
 Kent, OH 44242, USA. E-mail: {\tt nazarov@math.kent.edu}}
 \and
Yuval Peres
\thanks{Department of Mathematical Sciences,
Kent State University,
 Kent, OH 44242, USA. E-mail: {\tt yuval@yuvalperes.com}}
}

\date{\today}

\begin{document}

\maketitle

\begin{abstract}
We consider a boundary value problem for the $p$-Laplacian, posed in the exterior of  small cavities that all have the same $p$-capacity and are anchored to the unit sphere in $\mathbb{R}^d$, where $1<p<d.$
 We assume that the distance between anchoring points is at least $\eps$ and the characteristic  diameter of cavities is $\alpha \eps$, where $\alpha=\alpha(\eps)$ tends to 0 with $\eps$. We also assume that anchoring points  are asymptotically uniformly distributed as $\eps \downarrow 0$, and their number is asymptotic to   a positive constant times $\eps^{1-d}$.
 The solution $u=u^\eps$ is required  to be 1 on all cavities and decay to 0 at infinity.
 Our goal   is to describe the  behavior of solutions  for small $\eps>0$. We show that the problem possesses a critical window  characterized by
 $\tau:=\lim_{\eps \downarrow 0}\alpha /\alpha_c  \in (0,\infty)$, where   $\alpha_c=\eps^{1/\gamma}$ and $\gamma= \frac{d-p}{p-1}.$  We prove that outside
  the unit sphere, as $\eps\downarrow 0$, the solution converges to $A_*U$  for some   constant $A_*$, where   $U(x)=\min\{1,|x|^{-\gamma}\}$ is the radial $p$-harmonic function outside the unit ball.   Here the constant $A_*$ equals 0 if $\tau=0$, while $A_*=1$ if $\tau=\infty$.   In the critical window where $\tau$ is positive and finite, $ A_*\in(0,1)$  is explicitly computed in terms of the parameters of the problem.
 We also evaluate the limiting $p$-capacity in all three cases mentioned above. Our key new tool is the construction of an explicit ansatz function $u_{A_*}^\eps$ that approximates the solution $u^\eps$
 in $L^{\infty}(\mathbb{R}^d)$ and satisfies $\|\nabla u^\eps-\nabla u_{A_*}^\eps \|_{L^{p}(\mathbb{R}^d)} \to 0$  as $\eps \downarrow 0$.
 \end{abstract}

\medskip


\medskip

\noindent {\bf Mathematics Subject Classification MSC 2020:}  35J92,  31C45,  35B27,  35B40,  35J20,   35J25.
\medskip

%




\section{Introduction}
In this paper we consider  a boundary value problem for the $p$-Laplacian in a domain obtained from $\R^d$ by perforating it 
(i.e., removing small cavities) along the unit sphere.  The problem  is formulated as follows.
Given $\eps>0$, let $S=S(\eps)$ be a finite set of points   on the  unit sphere $\mathbb{S}^{d-1}\subset \mathbb{R}^d$ such  that  
 the Euclidean distance between any two points in $S$ is at least $\eps.$ Points in $S$ will be referred to as {\bf anchors}.
\begin{figure}[h]
\includegraphics[width=3.5 in]{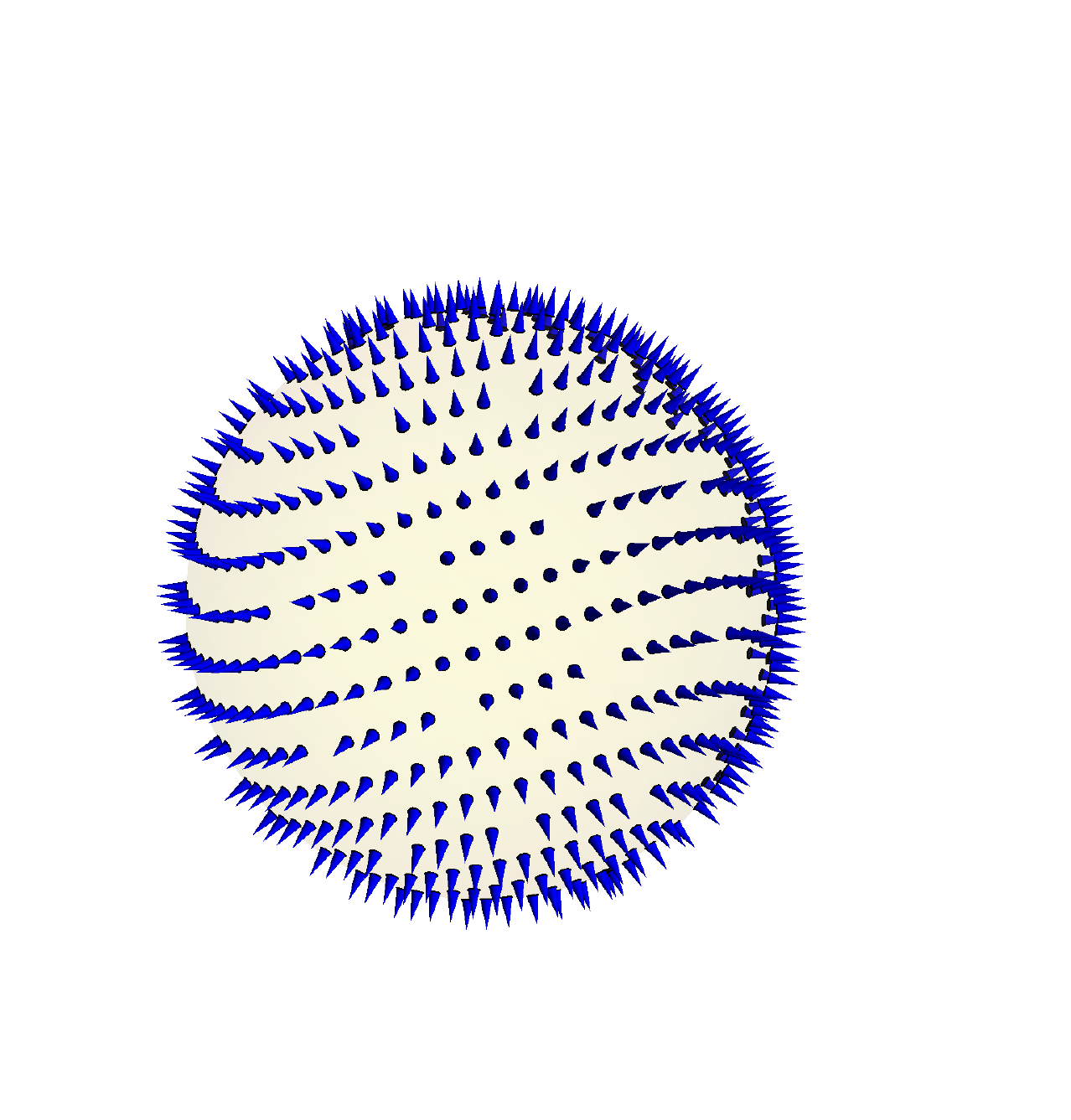}  \includegraphics[width=3.5 in]{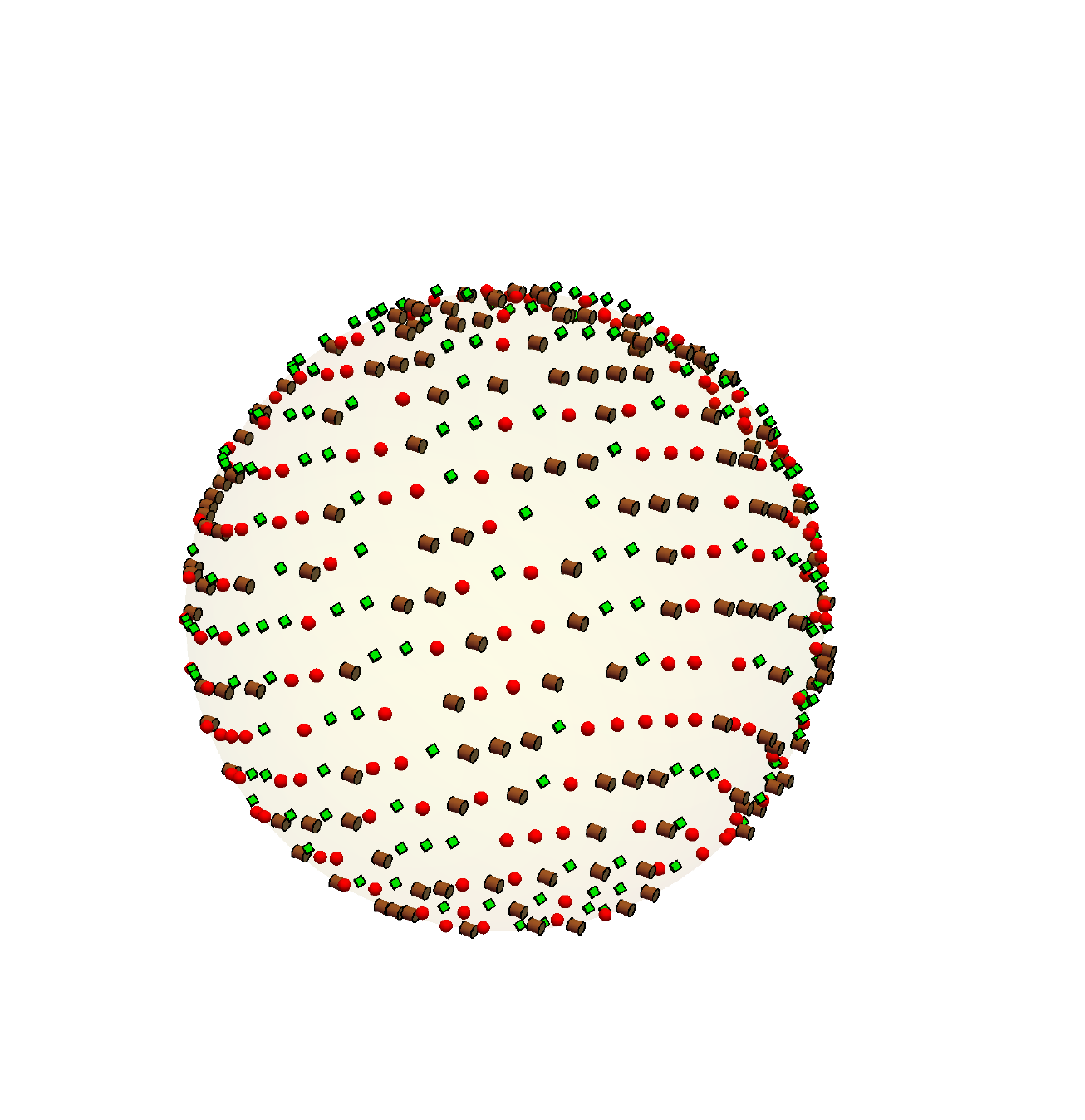}
\caption{
Two examples of sets $\Gamma$ in $\mathbb{R}^3$ that are unions of small cavities  anchored to the unit sphere $\mathbb{S}^2$:
congruent cones (left) and spheres, cubes and cylinders of the same $p$-capacity (right).
}
\label{fig:sphere}
\end{figure}
For each anchor $s$, let $K_s$ be a compact subset of the closed unit ball $\bar{B}(0,1) \subset \R^d$.
Let $0<\alpha=\alpha(\eps)\le \frac{1}{80}$  satisfy $\alpha(\eps) \to 0$ as $\eps \to 0$.
Define
\begin{eqnarray}\label{eq:1}
\Gamma=\Gamma_\eps:=\underset{s\in S} {\bigcup } (s+\alpha\eps K_s) \,.
\end{eqnarray}
 Two examples of such  sets are depicted in Figure \ref{fig:sphere}.

Assume that $1<p<d$ and  let $u=u^{\eps}$ be the Perron solution of the boundary value problem
\begin{eqnarray}\label{eq:2}
\left\{
\begin{array}{lll}
\Delta_p u= {\rm div}( |\nabla u|^{p-2} \nabla u)=0 &\mbox{in} & \mathbb{R}^d\setminus  \Gamma,\\
u=1 & \mbox{on} & \Gamma,\\
u(x) \to 0 & \mbox{as} & |{ x}| \to \infty.
\end{array}
\right.
\end{eqnarray}
(We recall the definition of Perron solutions in Section 2.) We are interested in the asymptotic behavior of $u^{\eps}$ as $\eps\to 0$ under the following  two
 \noindent{\bf key hypotheses:}
\begin{description}
\item  [(${\bf  H_1}$)]  The anchors are asymptotically equidistributed, that is
\begin{eqnarray}\label{eq:4}
\eps^{d-1}
 \sum_{s\in S(\eps)} \delta_s \overset{\ast} \rightharpoonup \sigma \mu \: ~ \mbox{as} ~ \eps\to 0,
\end{eqnarray}
where $\mu$ is the  uniform probability measure on the unit sphere $\mathbb{S}^{d-1}$
and   $\sigma>0$.
 \item [(${\bf  H_2}$)]   All the sets $K_s$  have the same {\bf ${\mathbf p}$-capacity}, i.e., there is some compact set $K \subseteq  \bar B(0,1)  \subset \R^d$
 such that ${\rm cap}_p(K_s)={\rm cap}_p(K)>0$ for all $\eps>0$ and $s \in S(\eps)$, where
\begin{eqnarray}
{\rm cap}_p(K ):=\inf\left\{ \int_{\mathbb{R}^d} \vert \nabla \psi \vert^p: ~\psi\in C_0^{\infty} (\mathbb{R}^d), ~ \psi\ge 1 ~\mbox{on} ~ K\right \},
\end{eqnarray}
see \cite[Chapter 2]{hkm}. %
A special case to keep in mind is when all the $K_s$ are rotated copies of $K$, as in the left part of Figure \ref{fig:sphere}.
\end{description}
The Perron solution $u$ of   \eqref{eq:2} is related to  the $p$-capacity of $\Gamma.$
Specifically, the following identity holds \cite[Theorems 9.33 and 9.35]{hkm}:
\begin{eqnarray}
{\rm cap}_p(\Gamma)= \int_{\mathbb{R}^d} \vert \nabla u \vert^p.
\end{eqnarray}
In this context, $u$ is called the $p$-equilibrium potential of $\Gamma.$
In previous works in similar setups \cites{KS14,KS16,GPPS} (summarized at the end of this section), asymptotics of $u^\eps$ in the Sobolev space $H^{1,p}$ were studied. The methods of the present paper enable us to obtain   precise $L^\infty$ asymptotics.

To motivate our results,  note that
since the $p$-capacity is sub-additive  \cite[Theorem 2.2]{hkm},  we have:
\begin{eqnarray} \label{eq:subadd}
  {\rm cap}_p(\Gamma )  \le  \underset{s\in S} \sum  {\rm cap}_p\bigl(s+\alpha\eps K_s\bigr)=
|S| (\alpha\eps)^{d-p} {\rm cap}_p( K)=   \frac{(\alpha\eps)^{d-p}}{\eps^{d-1}}  (\eps^{d-1} |S|) {\rm cap}_p( K) \,,
\end{eqnarray}
where we used the easily checked scaling relation
\begin{eqnarray}
{\rm cap}_p(a K)=a^{d-p} {\rm cap}_p(K), \qquad \forall  a>0.
\end{eqnarray}
Taking into account that $  \eps^{d-1} |S| \to \sigma $ as $\eps \to 0$, we obtain that
\begin{eqnarray}\label{eq:neweq1}
{\rm cap}_p(\Gamma )\le C \left(\frac{\alpha}{\alpha_c}\right)^{d-p} \,,
  \end{eqnarray}
  where $C$ is some constant independent of $\eps,$
 \begin{eqnarray}\label{eq:3}
 \alpha_c:=\eps^{\frac{1}{\gamma}} \,,
\end{eqnarray}
and
  \begin{eqnarray}\
\gamma:= \frac{d-p}{p-1} \,.
\end{eqnarray}

On the other hand, since $p$-capacity is monotone and $\Gamma \subset \bar B(0,1+\eps)$, we have
\begin{eqnarray} \label{eq:mono}
{\rm cap}_p(\Gamma )\le (1+\eps)^{d-p} {\rm cap}_p(\bar B(0,1)).
\end{eqnarray}
Thus, when the limiting value (as $\eps\to 0$) of $\frac{\alpha}{\alpha_c}$ increases from $0$ to $\infty$, the   natural upper  bounds \eqref{eq:neweq1} and \eqref{eq:mono}
for ${\rm cap}_p(\Gamma)$ cross each other. It turns out that in the same regime, the
solution $u$ of \eqref{eq:2} (outside of $\mathbb{S}^{d-1}$)
gradually transitions from $0$ to the $p$-equilibrium potential $U$ of the unit ball, defined as the solution of
\begin{eqnarray}\label{eq:U}
\left\{
\begin{array}{lll}
\Delta_p U=0 &\mbox{in} & \mathbb{R}^d\setminus \bar{B}(0,1),\\
U=1 & \mbox{on} & \bar{B}(0,1),\\
U(x) \to 0 & \mbox{as} & |x| \to \infty \,,
\end{array}
\right.
\end{eqnarray}
and given by
\begin{eqnarray}\label{eq:solU}
U(x):=\left\{
\begin{array}{ll}
1& 0\le |x| \le 1 \,,\\
|x|^{-\gamma} & |x|>1 \,.
\end{array}
\right.
\end{eqnarray}
The discussion above motivates our third hypothesis:
\begin{description}
  \item [(${\bf  H_3}$)] $\lim_{\eps \to 0} \alpha(\eps)  \eps^{-1/\gamma} =\tau \in [0,\infty]$.
\end{description}
Most of the paper will be devoted to the analysis of the critical window, where $0<\tau<\infty$.
Our first  result describes the asymptotics of  the capacity ${\rm cap}_p(\Gamma_\eps)$ and of the potential $u=u^\eps$ away from the unit sphere.
\begin{theorem}\label{t:1}
Suppose that hypotheses ${\rm (H_1),(H_2),(H_3)}$ hold.
  Then, as $ \eps\to 0$,
\begin{eqnarray} \label{eq:thm1A}
u^\eps(x)\to \left\{
\begin{array}{lll}
0& \mbox{if} & \tau=0 \,,\\
A_* U(x) & \mbox{if} &\  \tau\in(0,\infty) \,,\\
U(x) &  \mbox{if} & \tau= \infty \,
\end{array}
\right.
\end{eqnarray}
uniformly on compact subsets of  $\mathbb{R}^d\setminus \mathbb{S}^{d-1},$
where for $\tau \in [0,\infty)$,
\begin{eqnarray}\label{eq:A*}
A_*=
A_*(\tau)=\frac{\left(\sigma \tau^{d-p} {\rm cap}_p(K)\right)^\frac{1}{p-1}}{\left(\sigma \tau^{d-p} {\rm cap}_p(K)\right)^\frac{1}{p-1}+\left({\rm cap}_p(\bar B(0,1))\right)^\frac{1}{p-1}} \,.
\end{eqnarray}
Furthermore,
\begin{eqnarray} \label{eq:thm1B}
{\rm cap}_p(\Gamma_\eps)\to \left\{
\begin{array}{lll}
0& \mbox{if} & \tau=0 \,,\\
A_*^p {\rm cap}_p (\bar B(0,1))+(1-A_*)^p  {\rm cap}_p (K) \sigma \tau^{d-p} & \mbox{if} &\  \tau\in(0,\infty) \,,\\
{\rm cap}_p(\bar B(0,1))  &  \mbox{if} & \tau= \infty \, .
\end{array}
\right.
\end{eqnarray}
\end{theorem}
It will be convenient to define $A_*=1$ if $\tau=\infty$. If we assume  that $\alpha=C\eps^\zeta$ for  $\zeta>0$, then this theorem reveals a phase transition when the exponent $\zeta$ crosses   $1/\gamma$.

\medskip

\noindent{\bf Example}. To illustrate the Theorem, consider the simple special case where $d=3,p=2$ (so $\gamma=1$ and $\alpha_c=\eps$) and $K_s=\bar{B}(0,1)$ for all $s$.
Suppose that $\alpha=\tau\alpha_c$.
In this case, $U(x)=\min\{1, |x|^{-1}\}$ and $u^\eps(x)$ can be interpreted as the probability that a Brownian motion started from $x$ ever hits $\Gamma=\cup_{s \in S} B(s,\tau\eps^2)$.
The expression for $A_*$ simplifies to
 $A_*=\sigma \tau/(1+\sigma \tau)$.

 \medskip

\begin{figure}[h]
\centering \includegraphics[width=3.7in]{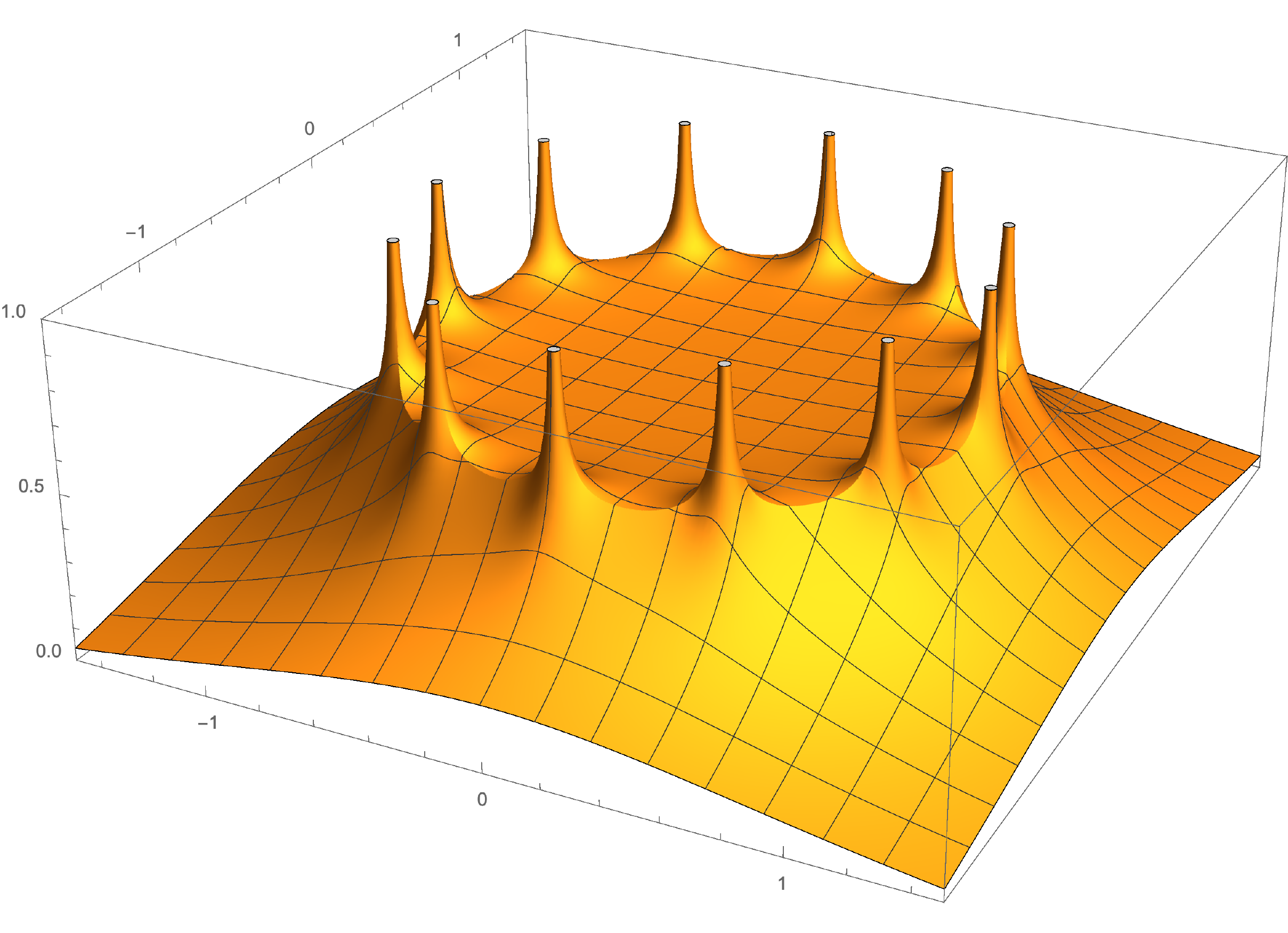} \qquad \includegraphics[width=2.5in]{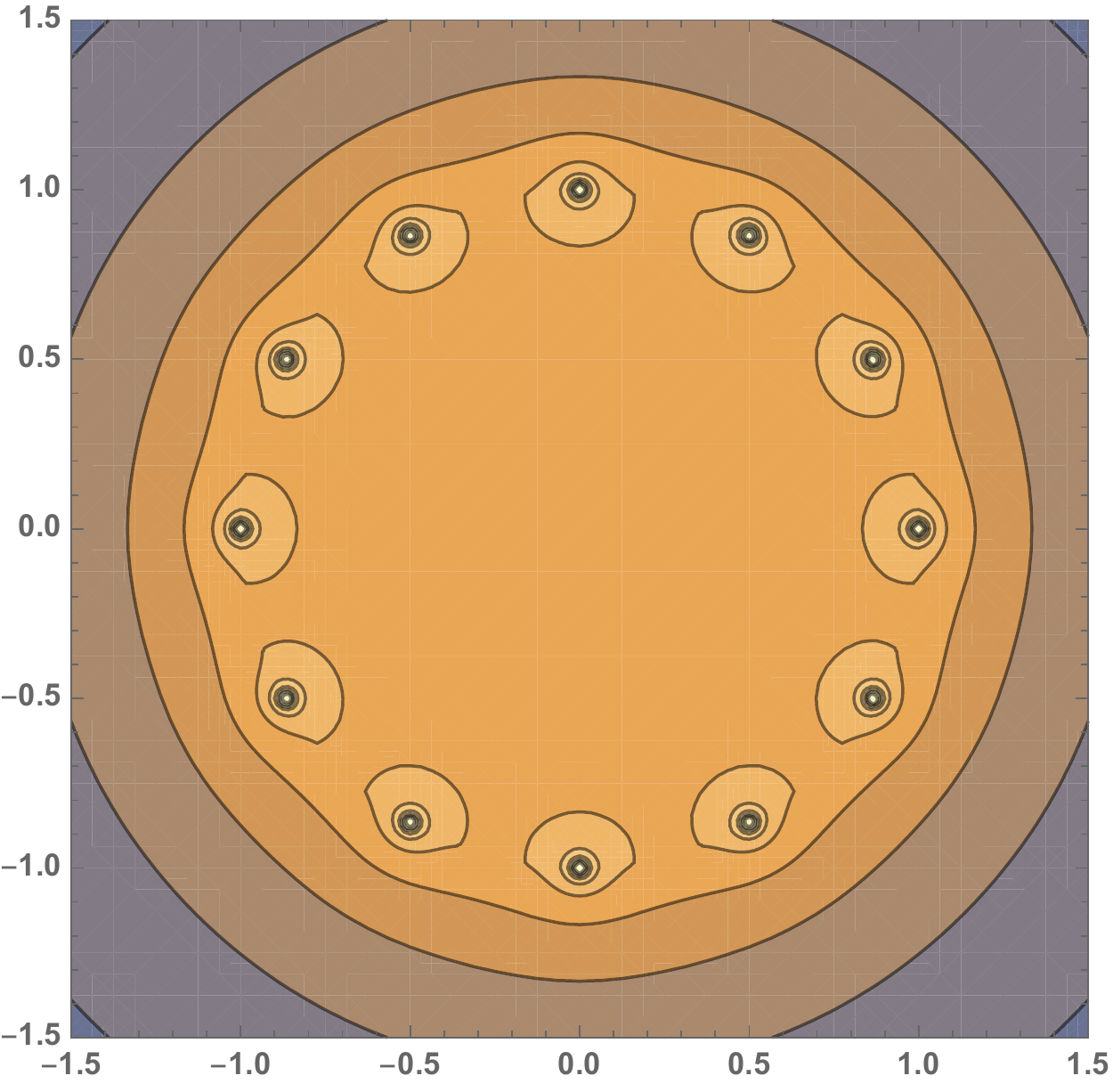}
\caption{The ansatz function $u_{A_*}$ and its level sets, determined by  $12$ equally spaced  anchors on the unit circle  ($d=2$).  Here $K=\bar B(0,1),$  $~p=1.5, ~\gamma=1, ~  \eps\approx 0.52, ~ \tau=1/40$ and  $A_* \approx 0.5$. }
\label{fig:ansatz}
\end{figure}

\begin{figure}[h]
\centering \includegraphics[width=2.15in]{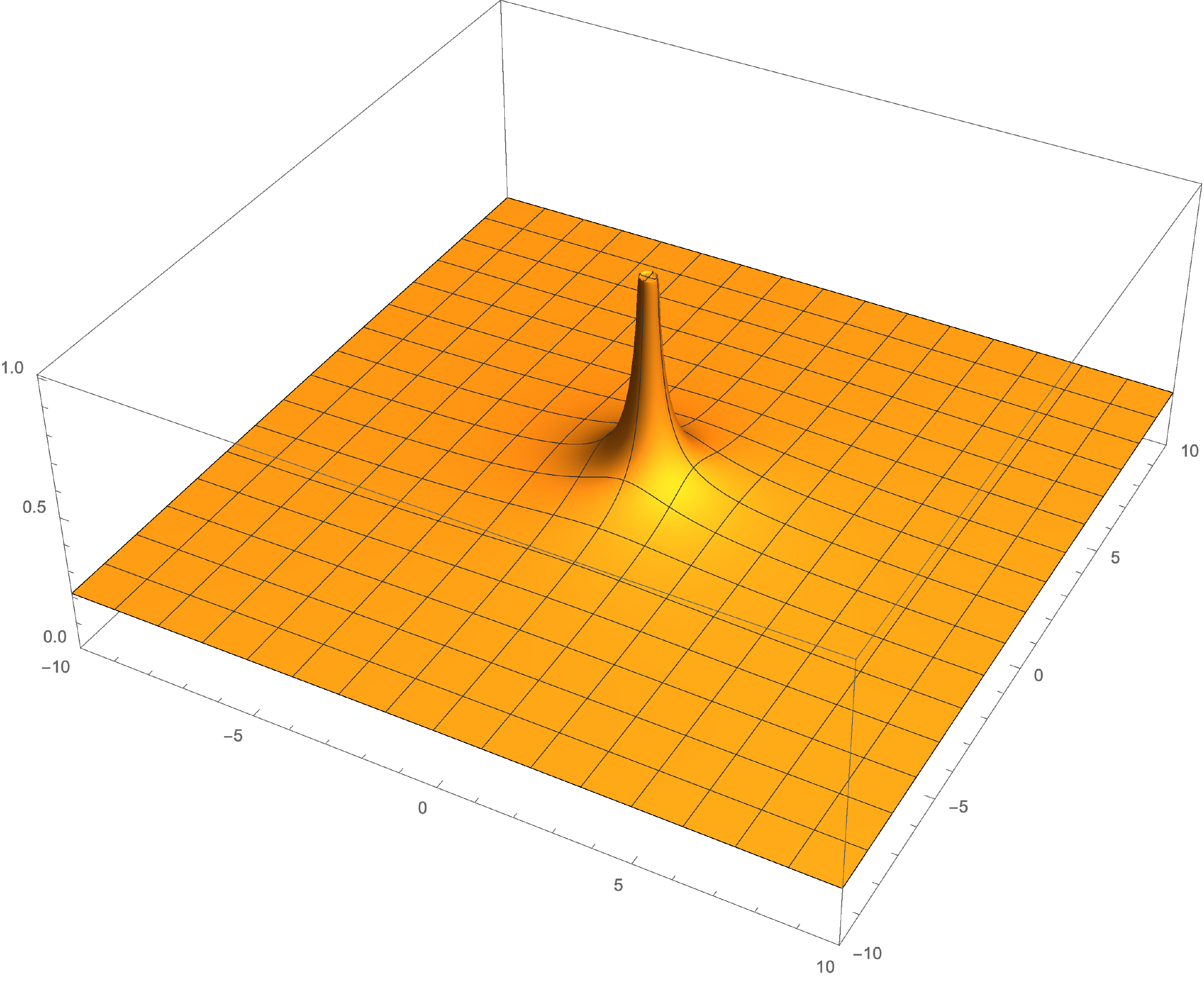}
\centering \includegraphics[width=2.15in]{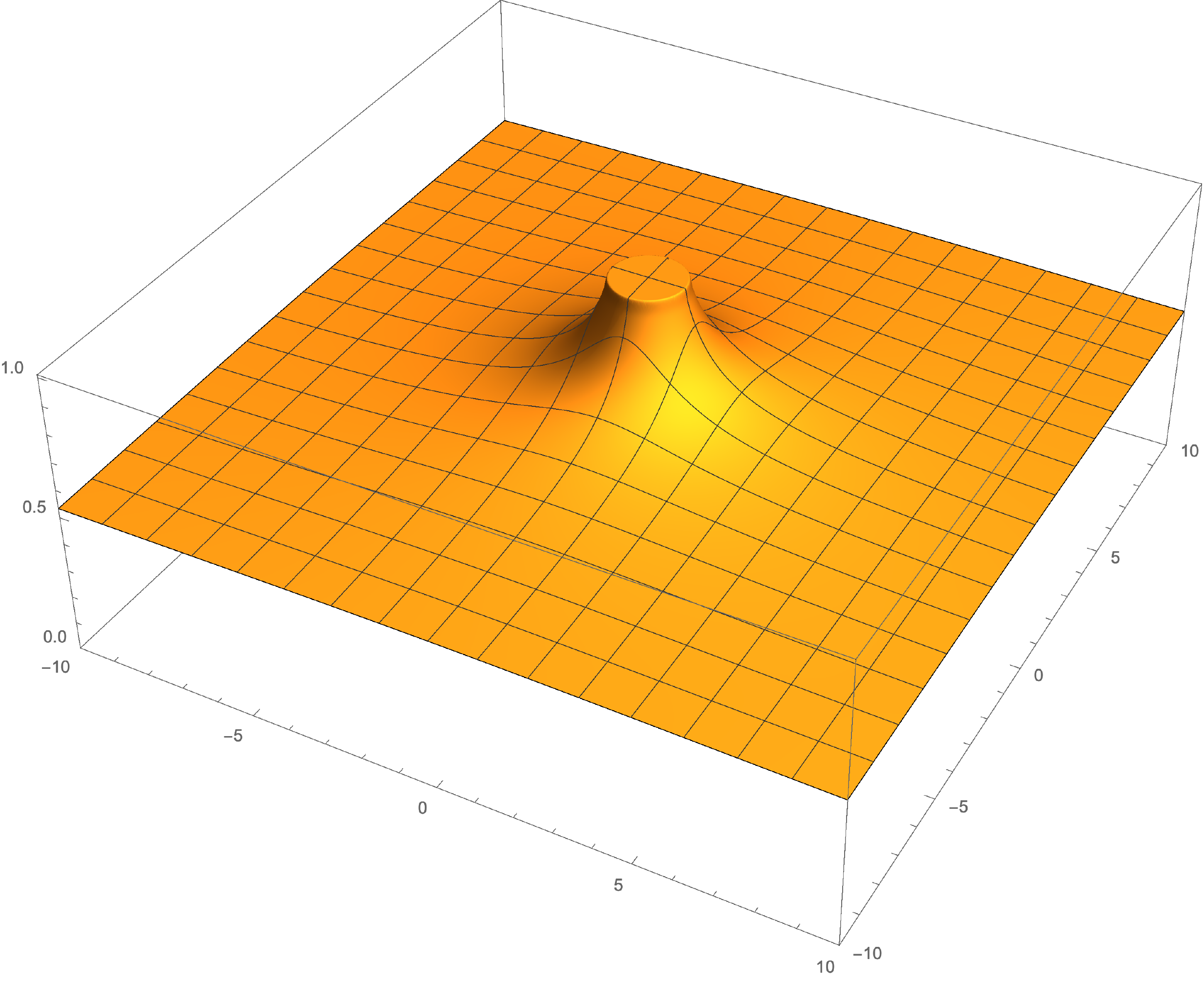}
\centering \includegraphics[width=2.15in]{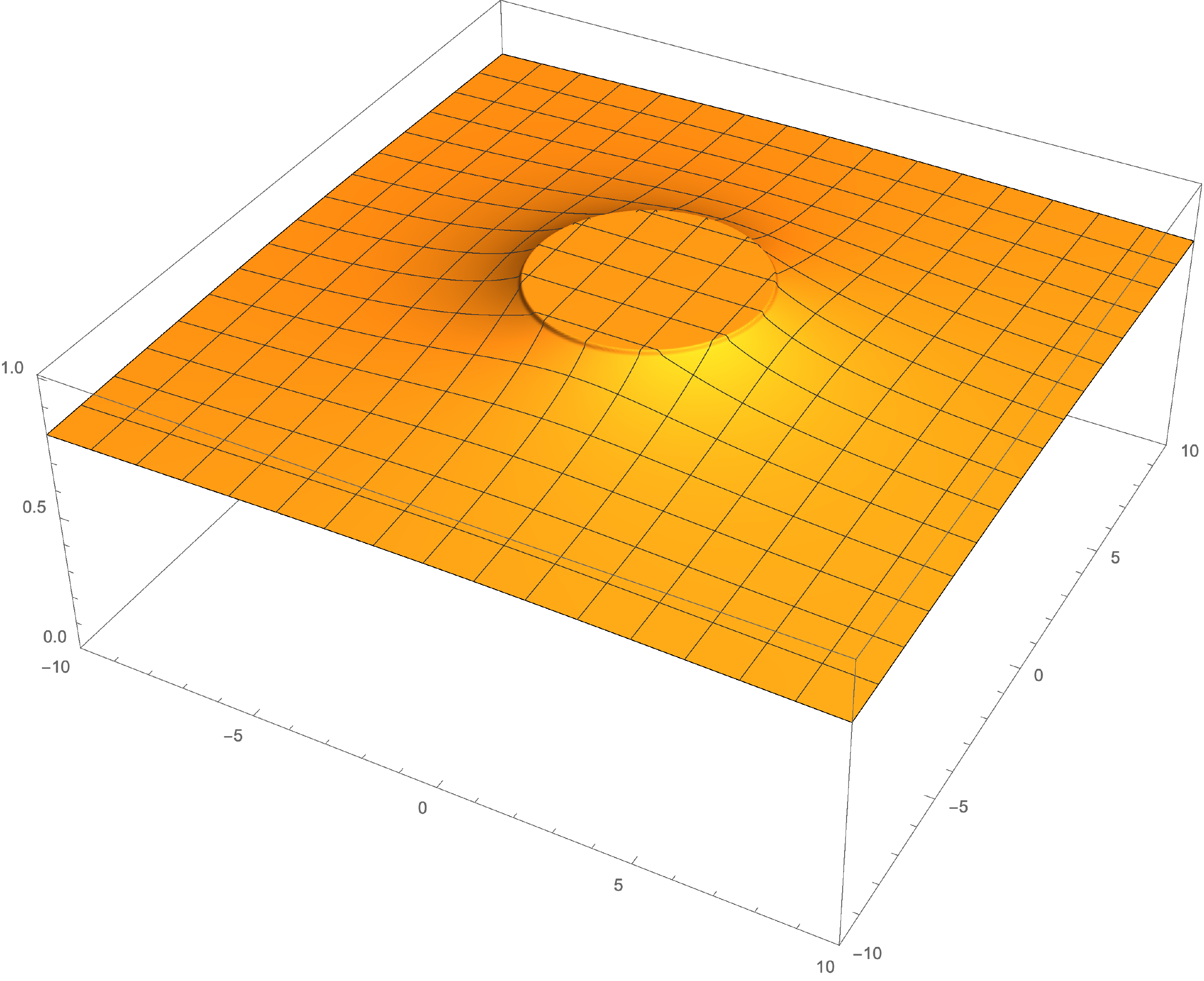}
\caption{The phase transition as $\tau$ increases: A single bump for $\tau=1/4$  ($A_* \approx 0.25$), $\tau=1$  ($A_* \approx 0.5$),  and $\tau=3$ ($A_*\approx 0.75$). Parameters:  $p=1.5,$ $ \gamma=1,$  $\eps \approx 5 \cdot 10^{-3}.$ The number of anchors is $1250$, only a square of side-length $20 \eps \alpha_c$ centered at one of them is depicted.  }
\label{fig:bump}
\end{figure}

For $\tau\in(0,\infty)$, the expression \eqref{eq:thm1A}  suggests that ${\rm cap}_p(\Gamma_{\eps})$  should approach    $A_*^p {\rm cap}_p(\bar B(0,1))$  as $\eps \to 0$.
This, however,  contradicts   the  actual limiting value of the capacity given by \eqref{eq:thm1B}.  Indeed, the limiting equilibrium potential only captures  the first summand in   \eqref{eq:thm1B},
 which accounts for the contribution from  the bulk. The second term in \eqref{eq:thm1B}, which accounts for the contribution of
  the equilibrium potential near the cavities,   completely disappears when looking at \eqref{eq:thm1A} alone.
 This  is similar to the ``term coming from nowhere'' discussed in \cite{CM97}.

The preceding observation raises the question:  How does the solution of \eqref{eq:2} behave  near the unit sphere?
To answer this question, we will introduce  the $p$-potential of $K_s$ in a ball $B(0,R)$, namely, the Perron solution $V_R^s:{\mathbb R}^d \to [0,1]$ of
\begin{eqnarray}\label{eq:V}
\left\{
\begin{array}{lll}
\Delta_p V_R^s=0 &\mbox{in} & B(0,R)\setminus  K_s,\\
V_R^s=1 & \mbox{on} & K_s,\\
V_R^s(x)= 0 & \mbox{for} & |x| \ge R.
\end{array}
\right.
\end{eqnarray}
We also define $V_\infty^s$ by replacing the last requirement in \eqref{eq:V} by $V_\infty^s(x)\to 0$ as $|x| \to \infty$. With $U$ defined by \eqref{eq:solU} and
$U_{1+\eps}(x):=U\Bigl( \frac{x}{1+\eps} \Bigr)$, we define, for each $A\in [0,1]$, the ansatz function
\begin{eqnarray}\label{eq:2da3}
u_{A}(x) := A U_{1+\eps}(x) +(1-A) \sum_{s\in S}V_{\frac{1}{10\alpha}}^s  \left(\frac{x-s}{\alpha \eps}\right)\,.
\end{eqnarray}

For $\tau<\infty$, the   convergence  $u^\eps(x) \to A_*U(x)$ as $\eps \to 0$ in Theorem \ref{t:1} does not hold uniformly in $\R^d \setminus {\mathbb S}^{d-1} $ because $\sup_{x \in  \R^d \setminus {\mathbb S}^{d-1}} u^\eps(x) =1$. The following theorem states that $\nabla u_{A_*}$ approximates $\nabla u$  in $L^p(\R^d)$ and  $u_{A_*}$ approximates $u$  in $L^\infty(\R^d)$  as $\eps \to 0$.
See Figures \ref{fig:ansatz} and \ref{fig:bump} for a depiction of $u_{A_*}$ and of how it changes when $\tau$ increases.
\begin{theorem} \label{main2}
  Suppose that hypotheses ${\rm (H_1),(H_2),(H_3)}$ hold. Then as $\eps\to 0$, we have
\begin{eqnarray}
\Vert \nabla u-\nabla u_{A_*}\Vert_{{L^p(\mathbb{R}^d)}}\to 0,
\end{eqnarray}
and
\begin{eqnarray}
\Vert u-u_{A_*}\Vert_{L^{\infty}(\mathbb{R}^d)}\to 0 \,.
\end{eqnarray}
\end{theorem}

\smallskip

\begin{remark} \label{tech}
Several technical challenges arise because there are no regularity assumptions on the compact sets $K_s$. First, the solution $u$ of  \eqref{eq:2} depends monotonically on $\alpha(\eps)$ only if $K_s$ are star-shaped. Moreover, in general $u$ need not be continuous on the boundary of $\Gamma$.
\end{remark}

\medskip
\noindent{\bf Related work.}
There is an extensive literature devoted to the $p$-Laplacian and nonlinear potential theory, see, e.g., the books   \cites{Mazya,Lindq,hkm,Adams-Hedberg} and the review paper \cite{pLrev}.
 Chapter 3 of \cite{MKhomo} discusses  homogenization problems for the $p$-Laplacian.
Most works on this topic   focus
 on the derivation of effective limiting equations in  domains that are perforated in the bulk; see, e.g., \cite{MKhomo,CM97}
 and references therein.  In the papers  \cites{KS14,KS16,GPPS}, which are most closely related to our study,  the authors considered homogenization
 for the $p$-Laplacian in domains where  the perforation  takes place only near a $(d-1)$ dimensional surface. In \cite{KS14} the authors assume that
  $p<1+d/2$ and  the cavities are obtained by intersecting a periodic structure with a hyperplane; the critical scaling $\alpha_c$ was already identified there
   under a mild regularity assumption on the cavities. In \cite{KS16}, the hyperplane is replaced by a convex surface and $p<1+d/4$. In \cite{GPPS} the 
   cavities are obtained by intersecting a periodic collection of small balls with the $\eps$-neighborhood of a smooth surface. The latter paper also contains a detailed review of earlier literature.
 In the works mentioned above, the authors determined the asymptotics of the $p$-capacity of the obstacle, and  the limiting behavior of the solution 
 in Sobolev space in the weak topology;  the behavior of the solution  near the perforated surface was not described. That is our main goal here.

\medskip

\noindent{\bf Roadmap.} While it is not so hard to show that the solution $u^\eps$ is generally well-behaved away from the cavities $s+\alpha\eps K_s$, with rare regions of high local energy, the main obstacle to obtaining our results (especially the $L^\infty$ estimates in Theorem  \ref{main2}), is ruling out such exceptional regions.

The rest of the paper is organized as follows. Section 2 presents some necessary background and preliminaries on $p$-potentials. In Section \ref{heart}, we bound the oscillation of $u^\eps$ in cones and, using the Besicovich covering lemma, obtain a lower bound for
 its energy in cones. In Section \ref{proofthm1},  we combine this lower bound with the global energy minimization property of $u^\eps$,
to infer a tight upper bound for its energy in cones; we  then use this  to deduce that $u^\eps$ is approximately constant on $\partial B(0,1+\delta)$ (where $\eps\ll \delta \ll1$), and prove Theorem \ref{t:1}.   Section \ref{necksec} shows that if $u^\eps$ is approximately constant on the boundary of a ball  $\partial B(0,1+\delta)$, then it is also approximately constant inside this ball, with the exception of small balls $B(s,\eps/10)$ around the anchors.
The comparison argument in that section is purely analytic, but is motivated by the connection with noisy tug of war games described in \cite{PS}. In the final section,  we combine the results from the preceding sections and derive Theorem   \ref{main2}.

\section{Preliminaries}

\subsection{Notation}
Given an open set $\Omega\subset \mathbb{R}^d$,  let $H^{1,p}(\Omega)$ denote the   Sobolev space of $L^p(\Omega)$ functions with weak gradient in $L^p(\Omega)$, and let
 $H^{1,p}_0(\Omega)$ denote the closure in $H^{1,p}(\Omega)$ of   $C_0^\infty(\Omega)$.
Write  $H_{loc}^{1,p}(\Omega)$ for the space of functions on $\Omega$ which, when restricted to every  subdomain $\Omega_1$ compactly contained in $\Omega$, are in  $H^{1,p}(\Omega_1)$.
\begin{definition}\cite[Chapters 6,7]{hkm}
Given an open set $\Omega\subset \mathbb{R}^d$, a continuous function $w\in H_{loc}^{1,p}(\Omega)$ is called $p$-{\bf harmonic} if $\Delta_p w=0$  weakly, i.e., if
$$ \int_\Omega |\nabla w|^{p-2} \nabla w  \cdot \nabla \eta=0 \,, \quad \forall \eta \in C_0^\infty(\Omega)\,. $$
A function $\psi:\Omega \to (-\infty,\infty] $ is $p$-{\bf superharmonic} if it is lower semi-continuous, not identically $\infty$ on any connected component of $\Omega$,
and satisfies the following comparison inequality on compactly contained subdomains $D \subset \Omega$: If $\psi \ge w$ on $\partial D$ and $w\in C(\bar{D})$ is $p$-harmonic in $D$, then $ \psi \ge w$ in $D$.
\end{definition}
The following principle, a special case of \cite[Proposition 7.6]{hkm}, will be used several times.
\begin{lemma}[Comparison principle]\label{l:comp}
Let $v,w$  be bounded $p$-harmonic
  functions in an open set $\Omega \subset \R^d$.
 If
 $$\limsup_{x \to y} v(x) \le \liminf_{x \to y} w(x)$$
 for all $y \in \partial \Omega$
 (including $y=\infty$ if $\Omega$ is unbounded)
  then $v \le w$  in $\Omega$.
  \end{lemma}
Next we give the definition of {\bf Perron solutions} from \cite[Chapter 9]{hkm}.
\begin{definition} \label{def:upper}
 Given a domain $\Omega \subset \mathbb{R}^d$ and boundary values $f: M \to \mathbb{R}$, where $\partial \Omega \subset M \subset \Omega^c$,
 the {\em upper class} $U_f^\Omega$ of $f$ consists of $p$-superharmonic functions $\psi:\Omega \to (-\infty,\infty] $ that are bounded below  and satisfy
\begin{eqnarray}\label{upperclass}
\liminf_{x \to y} \psi(x) \ge f(y), \quad \forall y \in \partial \Omega.
\end{eqnarray}
(Recall that if $\Omega$ is unbounded, then we include $\infty$ in $\partial \Omega$.)
The {\em upper Perron solution} $$h=\bar{H}_f^\Omega: \Omega \cup M \to [-\infty,\infty]$$ of  the boundary value problem
\begin{eqnarray}\label{eq:perron1}
\left\{
\begin{array}{lll}
\Delta_p h=0 &\mbox{in} & \Omega  \, ,\\
h(y)=f(y) & \mbox{for} & y \in M  \\
\end{array}
\right.
\end{eqnarray}
is defined in $\Omega$ by
\begin{eqnarray}\label{upperperron}
\bar{H}_f^\Omega(x):=\inf \{\psi(x) \, :\, \psi \in U_f^{\Omega} \}
\end{eqnarray}
and extended to agree with $f$ in $M$. The {\em lower Perron solution} of \eqref{eq:perron1} is defined by $\underline{H}_f^\Omega:= -\bar{H}_{-f}^\Omega$.

We say that $f$ is {\em resolutive} in $\Omega$ if the upper and lower Perron solutions of \eqref{eq:perron1} coincide; in this case we   refer to both of these simply as the  Perron solution.
\end{definition}
Theorem 9.25 in \cite{hkm} ensures that if $f$ is continuous on  $\partial \Omega$ and
 $(\mathbb{R}^d \setminus \Omega)$ contains a compact set of positive $p$-capacity, then $f$ is resolutive.

\medskip
  Following \cite{hkm}, we define the Dirichlet space
\begin{eqnarray}
L^{1,p}(\Omega):=\{ v\in H_{loc}^{1,p}(\Omega): ~ \nabla v\in L^p(\Omega)\},
\end{eqnarray}
and let $L_0^{1,p} (\Omega)$ be the  closure of $C_0^{\infty}(\Omega)$  in $L^{1,p}(\Omega)$ with respect to the semi-norm
$\left(\int_{\Omega} \vert \nabla v \vert^p\right)^\frac1p$, see \cite[Chapter 1.9, p. 13]{hkm}.

\begin{definition}
For  a function $v\in L^{1,p}(\Omega)$, we define its energy in a Borel set $B \subset \Omega$ by
\begin{eqnarray}
E(v,B) =\int_{B } \vert \nabla v \vert^p \,.
\end{eqnarray}
We write $E(v, {\mathbb R}^d)$ simply as $E(v)$.
\end{definition}

We now define the $p$-capacity (with respect to a domain $\Omega$) of a compact set $K \subset \Omega$ by
\begin{eqnarray} \label{defcap0}
{\rm cap}_p(K,\Omega):=\inf\{ E(\psi,\Omega), ~ \psi\ge 1 ~ \mbox{on} ~ K, ~ \psi\in C_0^{\infty}(\Omega) \}.
\end{eqnarray}
This definition readily implies that
\begin{eqnarray} \label{limcap}
{\rm cap}_p(K,B(0,R)) \downarrow {\rm cap}_p(K) \; {\rm as} \; R \uparrow \infty \,,
\end{eqnarray}
where
 $  {\rm cap}_p(K)={\rm cap}_p(K,{\mathbb R}^d)$.
In  \cite[p.\ 27-28]{hkm} it is shown that
\begin{eqnarray} \label{defcap1}
{\rm cap}_p(K,\Omega)=\inf\{ E(\psi,\Omega), ~ \psi\ge 1 ~ \mbox{on} ~ K, ~ \psi\in H_0^{1,p}(\Omega) \cap C(\Omega) \}.
\end{eqnarray}
Fix   $\phi \in C_0^{\infty}(\Omega)$ such that $\phi = 1$ on $K$, and define the admissible class
  \begin{eqnarray} \label{eq:bala-1}
{\mathcal A}(K, \Omega,\phi):= \{ \psi \in L_0^{1,p}(\Omega) ~\mbox{\rm such that}~  \psi-\phi\in L_0^{1,p}(\Omega\setminus K) \} \,.
\end{eqnarray}
Then we claim that
\begin{eqnarray} \label{eq:bala0}
{\rm cap}_p(K,\Omega) \le  E(\psi,\Omega) \;\;\mbox{\rm for}  \;  \psi \in  {\mathcal A}(K, \Omega,\phi)\,.
\end{eqnarray}
Indeed, suppose that $\psi \in L_0^{1,p}(\Omega)$ and $h_n \in C_0^\infty(\Omega \setminus K)$ satisfy $\nabla(\psi-\phi-h_n) \to 0$  in $L^p(\Omega \setminus K)$. Then   $\nabla(\phi+h_n)=\nabla\phi=0$ a.e.\ in $K$.  Therefore,
\eqref{defcap1} implies that $${\rm cap}_p(K,\Omega) \le E(\phi+h_n,  \Omega)=  E(\phi+h_n , \Omega \setminus K) \to E(\psi , \Omega \setminus K)   \,$$
as $n \to \infty$.

\smallskip

Next, let $\psi^K$ be  the Perron solution of
\begin{eqnarray}\label{eq:psi^K}
\left\{
\begin{array}{lll}
\Delta_p \psi^K=0 &\mbox{in} & \Omega\setminus  K \, ,\\
\psi^K=1 & \mbox{on} & K\, ,\\
\psi^K= 0 & \mbox{on} & \partial \Omega \, ,
\end{array}
\right.
\end{eqnarray}
where by convention, $\infty \in\partial \Omega$ if $\Omega$ is unbounded.
By \cite[Theorem 9.33]{hkm},   $\psi^K \in  {\mathcal A}(K, \Omega,\phi)$, and it coincides with the $p$-potential of $K$ defined in \cite[Chapter 6]{hkm}. Moreover,  $\psi^K$   is continuous and  $p$-harmonic on $\Omega \setminus K$, and it satisfies $\psi^K \equiv 1$ on $K$. It follows from Corollary 1.21 in \cite{hkm} that  $\nabla \psi^K(x)=0$ for almost every $x \in K$.
   By \cite[Theorem 9.35]{hkm},
$E(\psi^K, \Omega)= {\rm cap}_p(K,\Omega)$.
This yields an alternative definition of $p$-capacity: If $\phi \in C_0^{\infty}(\Omega)$ satisfies $\phi=1$ on $K$, then
\begin{eqnarray} \label{eq:bala}
{\rm cap}_p(K,\Omega):=\min \{ E(\psi,\Omega): ~ \psi \in  {\mathcal A}(K, \Omega,\phi) \}.
\end{eqnarray}
  By the strict convexity of the $L^p$ norm, $\psi^K$ is the unique  (up to translation by a constant) minimizer of the extremal problem \eqref{eq:bala}.

Thus  the Perron solutions $u,V_R^s$ and $U$, defined in \eqref{eq:2},  \eqref{eq:V} and \eqref{eq:U}, have the following properties:
\begin{eqnarray} \label{eq:musk}
\int_{\mathbb{R}^d} \vert \nabla u \vert^p ={\rm cap}_p(\Gamma), \quad \int_{B(0,R)} \vert \nabla V_R^s \vert^p ={\rm cap}_p(K_s, B(0,R)), \quad \int_{\mathbb{R}^d} \vert \nabla U \vert^p ={\rm cap}_p(\bar{B}(0,1))=\gamma^{p-1} \omega_{d-1} \,,
\end{eqnarray}
where $\omega_{d-1}$ is the surface area of the unit ball in $\mathbb{R}^d$, and the last equality is from \cite[Section 2.11]{hkm}.

\begin{remark} \label{pedant}
 Fix   $\phi \in C_0^{\infty}(\R^d)$ such that $\phi = 1$ on $\bar{B}(0,1)$ and $\phi=0$ outside $ B(0,2)$, and suppose that $R>2$.
  Since $V_R^s$ coincides with the $p$-potential of  $K_s$ in $B(0,R)$, \cite[Lemma 8.5]{hkm} yields that $V_R^s -\phi\in H_0^{1,p}(B(0,R) \setminus K_s)$,
  which readily implies that for any domain $\Omega$ that contains $B(0,R)$, we have
   $V_R^s \in H_0^{1,p}(\Omega   \setminus K_s$) and $E(V_R^s,\Omega)=E(V_R^s, B(0,R))$.
  While the classical gradient of $V_R^s$ on $\partial B(0,R)$ need not exist, the distributional gradient may be taken to be zero on this boundary,
   and in any case, does not affect the energy.
\end{remark}

\medskip
We will need an estimate on the rate of convergence in \eqref{limcap}.
\begin{lemma} \label{limcaprate}
Suppose that $K \subset \bar{B}(0,1) \subset \R^d$ is compact. Then
\begin{eqnarray} \label{balas}  {\rm cap}_p(K) \le {\rm cap}_p(K,B(0,R))   \le \bigl(1-R^{-\gamma}\bigr)^{-p}  {\rm cap}_p(K)\,.
\end{eqnarray}
\end{lemma}
\begin{proof}
 The first inequality in \eqref{balas} was already noted in \eqref{limcap}, so we focus on the second.
  Let $K_n$ be  the union of the closed dyadic cubes of side length $2^{-n}$ that intersect $K$.
  Denote by $\psi_n$ the $p$-potential of $K_n$ in $\R^d$ and let $\delta>0$.
  Since $K=\cap_{n=1}^\infty K_n$, Theorem 2.2(iv) in \cite {hkm} implies that there exists $n$ such that $K_n \subset B(0,1+\delta)$ and
    $${\rm cap}_p(K_n)\le {\rm cap}_p(K)+\delta \,.$$
   Note that $\psi_n$ is continuous in $\R^d$ since $\R^d\setminus K_n$ is regular by the corkscrew condition (See \cite[Theorems 6.27 and 6.31]{hkm}.)
    Recall   that $U$ given by \eqref{eq:solU} is the $p$-potential of the unit ball in $\R^d$.
    Then $\psi_n(x)  \le U\Bigl(\frac{x}{1+\delta}\Bigr)$ for all $x$   by the definition of upper Perron solutions,
    so the function $$\psi(x)=\max\Bigl\{0,\frac{\psi_n(x)-(1+\delta)^\gamma R^{-\gamma}}{1-(1+\delta)^\gamma R^{-\gamma}}\Bigr\}$$
    is in $H_0^{1,p}(B(0,R))$ by \cite[Lemmas 1.23 and 1.26]{hkm}.
     Thus, using $\psi$ in   \eqref{defcap1} yields $${\rm cap}_p(K,B(0,R)) \le E(\psi,B(0,R)) \le \bigl(1-(1+\delta)^{\gamma} R^{-\gamma}\bigr)^{-p}  {\rm cap}_p(K_n) \,,$$
     which implies  the second inequality of \eqref{balas} since $\delta>0$ is arbitrary.
\end{proof}

\subsection{Energy estimate on truncated cones}

\begin{definition}
Let  $y \in {\mathbb S}^{d-1}$ and let $Q_{\delta}(y):=B(y,\delta) \cap {\mathbb S}^{d-1}$ be the open  spherical cap of Euclidean radius $\delta$ centered at $y$.
We define the spherical cone
\begin{eqnarray}\label{a1}
\Lambda_{\delta}(y):=\{ r q: q\in Q_{\delta}(y)\; \, \text{\rm and} \; \, r>0\}.
\end{eqnarray}
\end{definition}

\begin{lemma}\label{l:holder}
Let $Q$ be an open set  on the unit sphere,  and for $R>1$ set
\begin{eqnarray}\label{eqh:3}
\Lambda(Q,R):=\{ r q: q\in Q\; \, \text{\rm and} \; \, r > R\}\,.
\end{eqnarray}
Suppose that   a nonnegative function $v\in L^{1,p}(\Lambda(Q,R)) \cap C(\overline{\Lambda(Q,R)})$ satisfies
 \begin{eqnarray}\label{eqh:7a}
 v(x)\to 0 \quad  \mbox{as}  \quad |x|\to\infty \,,
 \end{eqnarray}
 and
 \begin{eqnarray}\label{eqh:7b}
 v(Ry)\ge \tilde A>0, \quad \forall y\in Q\,.
 \end{eqnarray}
Then,
\begin{eqnarray}\label{eqh:7}
E(v,\Lambda(Q,R)) \ge E(v_{\tilde A},\Lambda(Q,R)) \,,
\end{eqnarray}
where  $v_{\tilde A}(x)=\tilde A R^\gamma U(x)= \tilde A   U(x/R)$ for all $x \in \Lambda(Q,R)$.

\end{lemma}

\begin{proof}
Assume \eqref{eqh:7a} and \eqref{eqh:7b} hold. Since $v\in L^{1,p}(\Lambda(Q,R))$, for $\mu$ almost every $y \in Q$,
the gradient $\nabla v(ry)$ exists and is finite for a.e. $r \in (R,\infty)$. Furthermore, \cite[Theorem 4.9.2]{EG92} implies that for $\mu$ almost every $y \in Q$, the function
$r\mapsto   v(ry) \,$  is  absolutely continuous in $\, (R,\infty)$.
 Thus, for  $\mu$ almost every $y \in Q$,    H{\"o}lder's inequality yields that
\begin{eqnarray}\label{eqh:1}
&& \tilde A\le \int_R^{\infty} \left | \nabla v(ry)\right| dr=\int_R^{\infty}  \left | \nabla v(ry)\right| r^\frac{d-1}{p} r^{\frac{1-d}{p}} \, dr\le \left(\int_R^{\infty}  \left | \nabla v(ry)\right|^p r^{d-1} \, dr\right)^\frac1p
\left( \int_R^{\infty} r^\frac{1-d}{p-1}dr\right)^\frac{p-1}{p}\nonumber \\
&&=\left(\frac{1}{\gamma R^{\gamma}}\right)^\frac{p-1}{p} \left(\int_R^{\infty}  \left | \nabla v(ry)\right|^p r^{d-1}dr\right)^\frac1p.
\end{eqnarray}
Taking $p$-th power and rearranging terms, we obtain
\begin{eqnarray}\label{eqh:2}
\tilde A^p R^{d-p}\gamma^{p-1} \le \int_R^{\infty}  \left | \nabla v(ry)\right|^p r^{d-1}dr.
\end{eqnarray}
Let   $\omega_{d-1}$ denote the area of the unit sphere.
Integrating \eqref{eqh:2} over $Q$, we obtain
\begin{eqnarray}\label{eqh:4}
\tilde A^p R^{d-p} \gamma^{p-1} \omega_{d-1} \mu (Q)  \le \omega_{d-1} \int_Q \int_R^{\infty}  \left | \nabla v(ry)\right|^p r^{d-1}\, dr \, d\mu(y)=\int_{\Lambda(Q,R)}   \left | \nabla v(x) \right|^p \, dx =E(v,\Lambda(Q,R))\,.
\end{eqnarray}

By the condition for equality in H{\"o}lder's inequality, the second inequality in \eqref{eqh:1} is an equality if for some constant $C$,
\begin{eqnarray}\label{eqh:5}
 \Bigl[|\nabla v(ry)| r^{\frac{d-1}{p}}\Bigr]^{p-1}=C r^{\frac{1-d}{p}} \quad \mbox{\rm for a.e.} \; r \ge R \,.
\end{eqnarray}
The function $v_{\tilde A}(x)=\tilde A R^\gamma U(x)$ satisfies  \eqref{eqh:5} and the boundary condition $v_{\tilde A}(Ry) =\tilde A$ for all $y \in Q$. Since
$$ |\nabla v_{\tilde A}(ry)|= -\frac{d}{dr}  v_{\tilde A}(ry)
$$
for $y \in \partial B(0,1)$, the first inequality in  \eqref{eqh:1} is also an equality if $v=v_{\tilde A}$. Therefore,
\begin{eqnarray}\label{eqh:6}
E(v_{\tilde A},\Lambda(Q,R)) =\tilde A^pR^{d-p} \gamma^{p-1} \omega_{d-1} \mu (Q) \,.
\end{eqnarray}
Combining \eqref{eqh:4} and \eqref{eqh:6}, we obtain   \eqref{eqh:7}.
\end{proof}

\subsection{Energy of an ansatz function}

Recall from \eqref{eq:2da3}  the definition
$$ u_{A}(x) = A   U_{1+\eps}(x) +(1-A) \sum_{s\in S} V_{\frac{1}{10\alpha}}^s\left(\frac{x-s}{\alpha \eps}\right) \,.
$$
 Since the gradients of the summands on the right-hand side have disjoint supports, we have
\begin{eqnarray} \label{decomp-ener}
E(u_A)=A^pE(U_{1+\eps})+(1-A)^p \sum_{s \in S}  E\Bigl( x \mapsto V_{\frac{1}{10\alpha}}^s \Bigl(\frac{x-s}{\alpha \eps}\Bigr)\Bigr) \,.
\end{eqnarray}
(By Remark \ref{pedant}, we can compute the energy over all of $\R^d$ in the summands on the right-hand side.)
 For any function $\psi \in L^{1,p}({\mathbb R}^d)$, vector $s \in {\mathbb R}^d$ and scalar $r>0$,  scaling arguments yield that
\begin{eqnarray} \label{scale}
E \Bigl(x \mapsto \psi\Bigl(\frac{x-s}{r}\Bigr) \Bigr)=r^{d-p} E(\psi) \,.
\end{eqnarray}
Therefore,
\begin{eqnarray}
E(u_A)=A^p (1+\eps)^{d-p} E(U)+(1-A)^p \sum_{s \in S} (\alpha\eps)^{d-p}E\bigl(V_{\frac{1}{10\alpha}}^s\bigr).
\end{eqnarray}
Using   the fact that $E(U)={\rm cap}_p(\bar B(0,1))$, we obtain
\begin{eqnarray}\label{eq:separ}
E(u_A)=(1+\eps)^{d-p} A^p {\rm cap}_p(\bar B(0,1))+(1-A)^p  \sum_{s \in S}  \kappa_s \,,
\end{eqnarray}
where
\begin{eqnarray} \label{decomp-ener3}
\kappa_s:=  {\rm cap}_p\left(\alpha \eps K_s, B\left(s,\frac{\eps}{10}\right)\right) \,.
\end{eqnarray}
By Lemma \ref{limcaprate} applied to $K_s$ and hypothesis   ${\rm (H_2)}$ on equality of capacities, we infer that
\begin{eqnarray}\label{eq:separ1}
 \Bigl|{\rm cap}_p\Bigl(K_s,   B\Bigl(0,\frac{1}{10\alpha}\Bigr)\Bigr) - {\rm cap}_p (K)\Bigr| \le  J(\alpha){\rm cap}_p (K) \,,
\end{eqnarray}
where
\begin{eqnarray}\label{eq:J}
 J(\alpha):=\bigl(1-(10\alpha)^\gamma\bigr)^{-p}-1 \to 0 \; \; \mbox{\rm as} \; \; \eps \to 0 \,.
 \end{eqnarray}
 Therefore,
 \begin{eqnarray}\label{eq:separ2}
\Bigl|E(u_A)-(1+\eps)^{d-p} A^p {\rm cap}_p(\bar B(0,1))-(1-A)^p (\alpha\eps)^{d-p}|S| {\rm cap}_p (K)\Bigr| \le
(1-A)^p  (\alpha\eps)^{d-p}  |S|J(\alpha){\rm cap}_p (K) \,.
\end{eqnarray}
Next, as $\eps \to 0$, we have by \eqref{eq:3} and the definition of $\tau$ in ${\rm (H_3)}$ that
\begin{eqnarray}
 (\alpha\eps)^{d-p}|S|= \bigl(\alpha \eps^{-1/\gamma}\bigr)^{d-p} \eps^{d-1} |S| \to \tau^{d-p}\sigma \,.
\end{eqnarray}
  In conjunction with  \eqref{eq:separ2}, this gives
\begin{eqnarray} \label{limener}
E(u_A)\to \varphi_\tau(A) \; \mbox{\rm as} \; \eps\to 0 \,,
\end{eqnarray}
where
\begin{eqnarray}
\varphi_\tau(A)=A^p {\rm cap}_p (\bar B(0,1))+(1-A)^p  {\rm cap}_p (K) \sigma \tau^{d-p}.
\end{eqnarray}
If $\tau \in [0,\infty)$, then the convergence in \eqref{limener} is uniform in $A \in [0,1]$.
Note that $\varphi_\infty(1)={\rm cap}_p (\bar B(0,1))$ and $\varphi_\infty(A)=\infty$ for all $A \in [0,1)$,
 and in the latter case, the convergence in \eqref{limener} means that the left-hand side tends to $\infty$.
   For $\tau \in (0,\infty)$, the function $\varphi_\tau$ is a continuous strictly convex function on $[0,1]$. Differentiation shows that it attains
    its minimum  at $A_*=A_*(\tau)$ given by \eqref{eq:A*}, with
\begin{eqnarray} \label{phimin}
 \varphi_\tau(A_*)= \left(\frac{\left(\sigma \tau^{d-p} {\rm cap}_p(K){\rm cap}_p(\bar B(0,1))\right)^\frac{1}{p-1}}{\left(\sigma \tau^{d-p}
 {\rm cap}_p(K)\right)^\frac{1}{p-1}+\left({\rm cap}_p(\bar B(0,1))\right)^\frac{1}{p-1}}\right)^{p-1} .
\end{eqnarray}
Clearly, $\varphi_\tau(A_*) \le \varphi_\tau(1) ={\rm cap}_p (\bar B(0,1))$ for all $\tau \in [0,\infty]$.

\begin{remark} \label{rem:adm}
Let $\eta:{\mathbb R}^d \to [0,1]$ be a $C^\infty$ cutoff function such that $\eta \equiv 1$ on $B(0,1)$ and $\eta \equiv 0$ outside $B(0,2)$. Then for $R>2$, we have that $V_R^s-\eta \in L_0^{1,p} (B(0,R)\setminus K_s)$ for each $s \in S$.
Thus, for every anchor $s$,
$$V_{\frac{1}{10\alpha}}^s\left(\frac{x-s}{\alpha \eps}\right)-\eta\ \left(\frac{x-s}{\alpha \eps}\right) \in L_0^{1,p} \Bigl(B(s,\eps/10)\setminus (s+\alpha\eps K_s)\Bigr) \,.
$$
Consequently, for all $A \in [0,1]$,  the ansatz function $u_A$ given by \eqref{eq:2da3}  belongs to the admissible class ${\cal A}(\Gamma,\mathbb{R}^d,\phi)$ with
$$\phi(x):=A\eta\Bigl(\frac{x}{1+\eps}\Bigr)+(1-A)\sum_{s\in S} \eta \Bigl(\frac{x-s}{\alpha \eps}\Bigr) \,.$$
 By \eqref{eq:bala0}, this implies that  $E(u)\le E(u_A).$ 
\end{remark}

\begin{cor} \label{cor:en}
For every $y$ on the unit sphere, $\delta>0$ and $\tau \in [0,\infty]$, we have
\begin{eqnarray}\label{c13b}
E(u_A,\Lambda_\delta(y)) \to \mu(Q_\delta) \varphi_\tau(A) \quad \mbox{as} \quad \eps \to 0 \,,
\end{eqnarray}
and the convergence is uniform in $A \in [0,1]$ and in $y \in {\mathbb S}^{d-1}$, provided that $\tau<\infty$.
Moreover, for every $\delta \in (0,1/2)$ and sufficiently small $\eps$ (that may depend on $\delta$),
\begin{eqnarray}\label{eq:bady3}
\forall A \in [0,1], \qquad E(u_A, \Lambda_{  \delta}(y)) \ge (1- \delta)^p E(u_{A^*}, \Lambda_{ \delta}(y))  \; \; \mbox{\rm for} \;\;  y \in {\mathbb S}^{d-1} \,.
\end{eqnarray}
\end{cor}
\begin{proof}
By the definition of $u_A$ and \eqref{eq:musk}, we have
\begin{eqnarray} \label{decomp-ener2}
(1-A)^p \sum_{s \in S \cap Q_{\delta-\eps}(y)} \kappa_s \le E(u_A, \Lambda_{\delta}(y))- A^p (1+\eps)^{d-p}E(U, \Lambda_{\delta}(y) ) \le
 (1-A)^p \sum_{s \in S \cap Q_{\delta+\eps}(y)} \kappa_s \,,
\end{eqnarray}
where $\kappa_s$ was defined in \eqref{decomp-ener3}.
Since $U$ is radial,
 $$E(U, \Lambda_{\delta}(y) )=\mu(Q_\delta)  {\rm cap}_p\bigl(\bar{B}(0,1)\bigr) \,.$$
 The rest of the proof of \eqref{c13b}  proceeds exactly like the proof of \eqref{limener},
 since the left-hand and right-hand sides of \eqref{decomp-ener2} have the same asymptotics:
 $$\eps^{d-1} |S \cap Q_{\delta\pm\eps}(y)|  \to \sigma \mu(Q_\delta) \quad \mbox{as} \quad \eps \to 0 \,.
 $$

It remains to verify \eqref{eq:bady3}. If $0<\tau<\infty$, then this follows from the uniform convergence (in the parameter $A$) in \eqref{c13b}.
 If $\tau=0$, then $A_*=0$. In this case, \eqref{eq:bady3}
is obvious if  $A\le \delta$.  On the other hand, if $\tau=0$ and $A>\delta$, then the left-hand side of  \eqref{eq:bady3} is at least $\delta^p E(U,\Lambda_\delta(y))>0$ which does not depend on $\eps$,
while the right-hand side of  \eqref{eq:bady3} tends to $0$ as $\eps \downarrow 0$; Thus, \eqref{eq:bady3}  holds in this case as well provided $\eps$ is small enough.

Finally,  if $\tau=\infty$, then $A_*=1$ and
$u_{A^*}=U_{1+\eps}$;  in this case, \eqref{eq:bady3}
is obvious if $A>1-\delta $, while the lower bound  $$E(u_A, \Lambda_{\delta}(y))\ge
 (1-A)^p \sum_{s \in S \cap Q_{\delta-\eps}(y)} \kappa_s \,,$$
 implies  that
 $$\lim_{\eps \to 0} \min_{A \in [0,1-\delta]} E(u_A, \Lambda_{\delta}(y))  = \infty \,,$$
so \eqref{eq:bady3}  holds if  $\tau=\infty$ and $A\le 1-\delta $ as well.
\end{proof}

\section{Bounding oscillation and energy of $u$ in cones} \label{heart}
In this section, the positive constants $C,C_1,C_2,\ldots$   depend only on $d,p$.
\begin{lemma} \label{l:1}
Suppose that for some $r>0$ and $z \in \R^d$ we have
\begin{eqnarray}\label{a4}
B(z,r)\subset B(z,5r/4) \subset \R^d\setminus \Gamma
\end{eqnarray}
and for some
$\lambda>0, \beta \ge 0$, the solution $u$ of \eqref{eq:2} satisfies
\begin{eqnarray}\label{a3}
E(u, B(z,5r/4)) \le \lambda^p r^{d-1-\beta} \,.
\end{eqnarray}
Then,
\begin{eqnarray}\label{a5}
\underset{ B(z,r)} {\rm osc}~ u  \le C_1 \lambda r^\frac{p-1-\beta}{p} \,,
\end{eqnarray}
where
\begin{eqnarray}\label{a5a}
\underset{ D} {\rm osc}~ u:=\sup_{D} u-\inf_{D} u
\end{eqnarray}
stands for the  oscillation over the set $D$.
\end{lemma}

\smallskip

\begin{proof}
By Poincar{\' e}'s inequality, there exists a real $t$ such that
\begin{eqnarray}\label{a6}
J_p=\int_{B(z,5r/4)} \left  \vert u - t  \right \vert^p\le C_2 r^p  \int_{B(z,5r/4)} \vert \nabla u \vert^p =C_2 r^p E(u, B(z,5r/4))\,.
\end{eqnarray}
The hypothesis \eqref{a3} and the inequality \eqref{a6} give
\begin{eqnarray}\label{a7}
J_p\le C_2\lambda^p r^{p+d-1-\beta}\,.
\end{eqnarray}
Let
\begin{eqnarray}
u_1(x)=\max\{u(x)-t,0\}, \qquad u_2(x)=\max\{t-u(x),0\}.
\end{eqnarray}

By \cite[ Lemma 3.6]{Lindq} we  have
\begin{eqnarray}\label{a8}
\left \Vert u_i \right \Vert_{L^{\infty}(B(z,r))}
\le C_3 \left(\frac{J_p}{r^d}\right)^\frac{1}{p} \le C_4 \lambda r^\frac{p-1-\beta}{p}, \qquad i=1,2,
\end{eqnarray}
which implies \eqref{a5} with $C_1=2C_4$.
\end{proof}

\begin{definition}\label{defbad}
Fix $\beta =\frac{p-1}{2p}$.  Given $\delta<1/20$ and $\eps<\delta/20$, an anchor $s \in S$ will be called  a {\bf good anchor} if
\begin{eqnarray}
E(u,\Lambda_{\zeta}(s))\le \zeta^{d-1-\beta},\qquad \forall \zeta\in[\eps,\delta] \,.
\end{eqnarray}
Otherwise, $s$ will be called a {\bf bad anchor}.
\end{definition}

\begin{figure}[h!]
\centering
\includegraphics[width=6.5in]{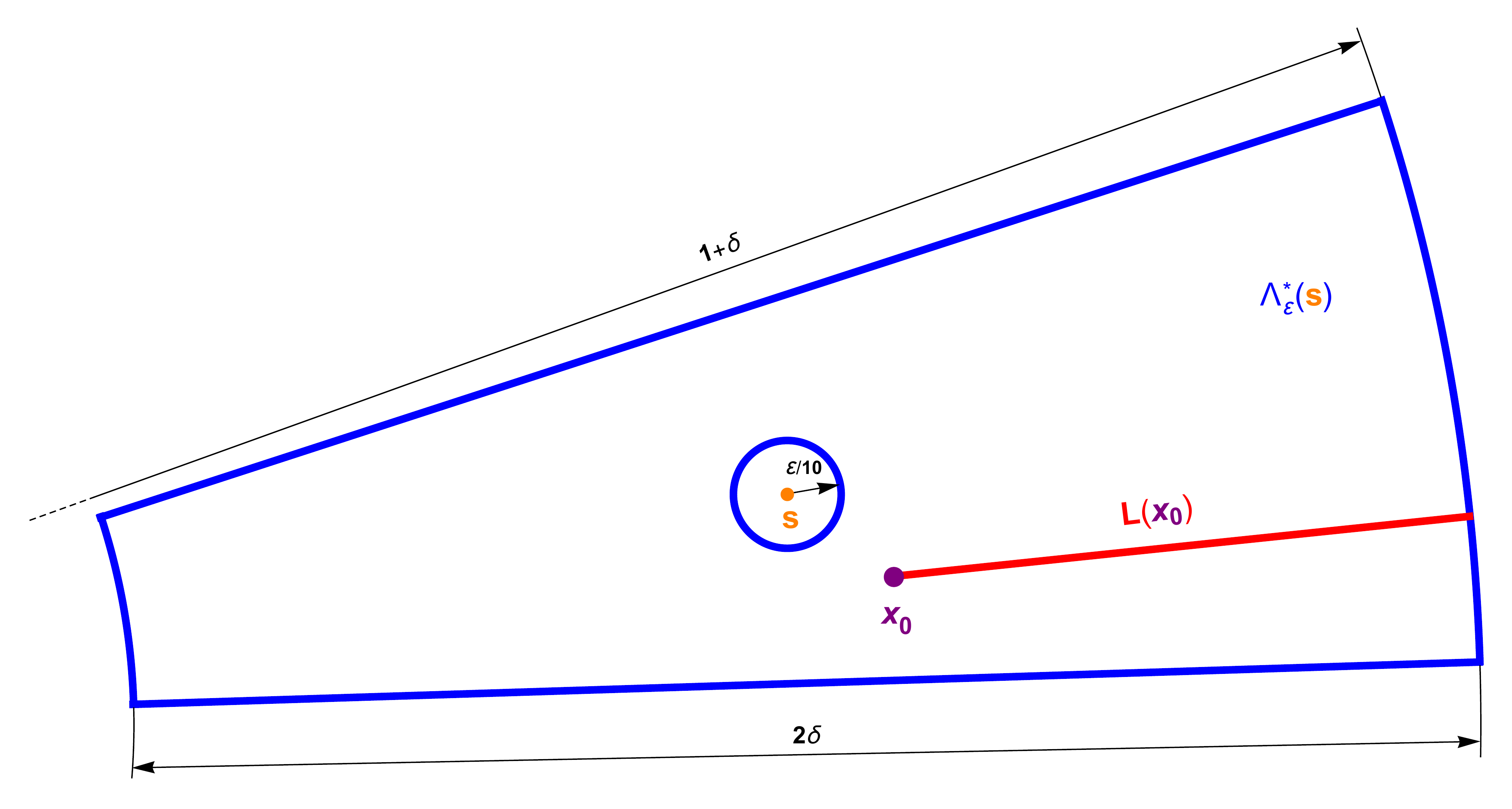}\\
\includegraphics[width=6.5in]{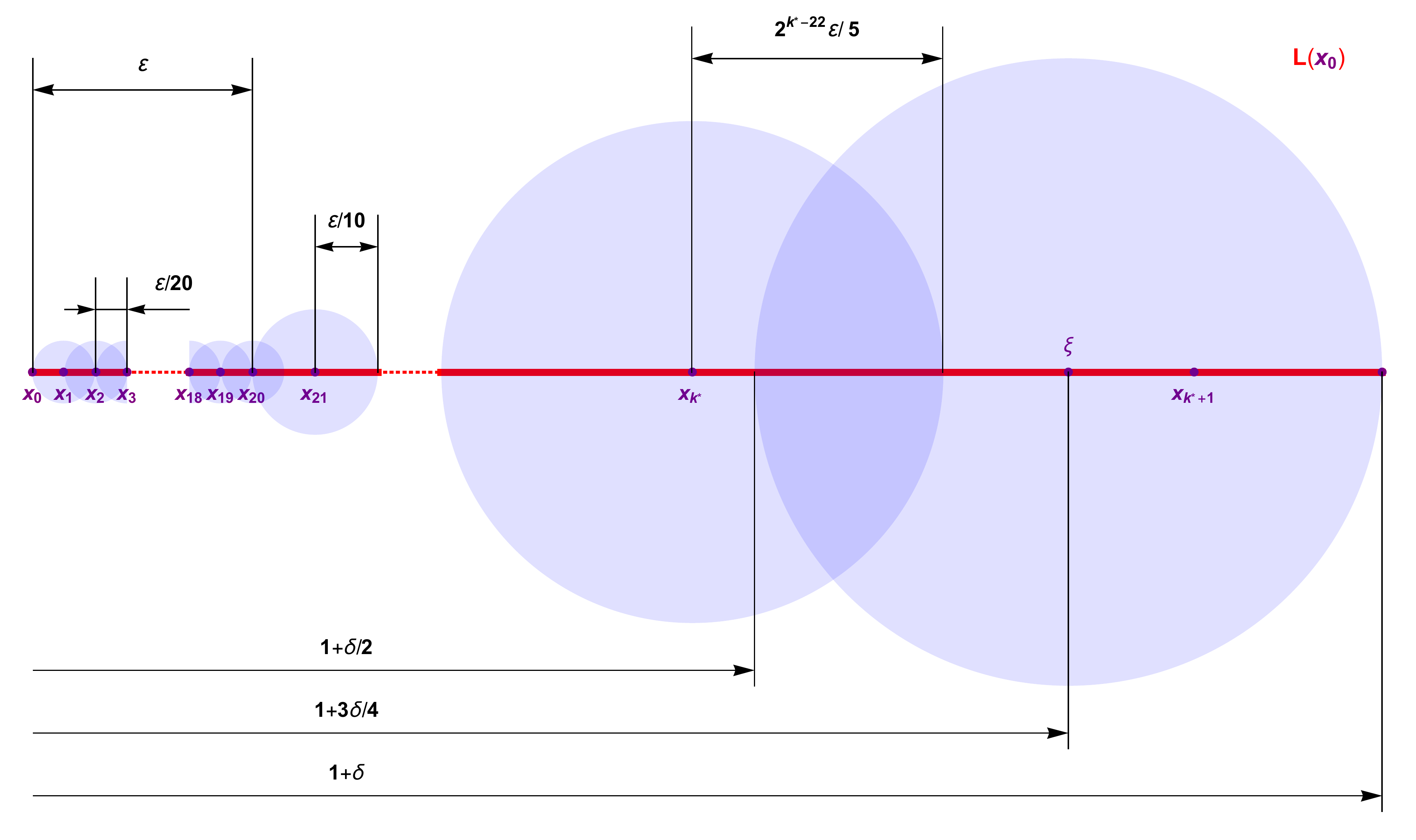}
\caption{Geometric objects used in the proof of Lemma \ref{l:2}: The top figure indicates  the point $x_0$ and the line segment $L(x_0)$   in the cored wedge  $\Lambda_{\eps}^*(s)$ defined by \eqref{eq:star}.
The bottom figure shows the points $x_k$ and   $\xi$ on  $L(x_0)$ and the  corresponding overlapping balls centered at these points.}
\label{fig:lballs}
\end{figure}

\begin{lemma}\label{l:2}

Fix $\eps,\delta,\beta$ as in the definition above.  Suppose $s$ is a good anchor and
let
\begin{eqnarray} \label{eq:star}
\Lambda_{\eps}^*(s):= \Bigl\{ x\in \Lambda_{\eps/2}(s)\setminus B\left(s,\frac{\eps}{10}\right) : 1-\delta \le \vert x \vert \le 1+\delta \Bigr\} \,.
\end{eqnarray}
Then there exists $C_5$, such that
\begin{eqnarray}
\underset{\Lambda_{\eps}^*(s) } {\rm osc} ~ u \le C_5 \delta^\frac{p-1-\beta}{p} \,.
\end{eqnarray}
\end{lemma}

\begin{proof}
Take $x_0\in \Lambda_{\eps}^*(s)$ and connect $x_0$ radially by a line segment $L(x_0)$  to the cap   $\Lambda_{\eps}^*(s) \cap \{|x|=1+\delta\}$ or   the cap   $\Lambda_{\eps}^*(s) \cap \{|x|=1-\delta\}$, depending on whether $|x_0| \ge 1$ or not.  (Observe that the oscillation of $u$ on each of these caps is at most $C_1 \delta^\frac{p-1-\beta}{p}$ by Lemma~\ref{l:1}.)
Assume first that $1 \le |x_0| \le 1+\delta/4.$  We then define a sequence of points   $\{x_k\}_{k \ge 0}$ along $L(x_0)$ and corresponding radii $r_k$  as follows  (See Figure \ref{fig:lballs}):

\smallskip

 For $k=0,1,\ldots, 20$, let $|x_k|=|x_0|+k\eps/20$ and $r_k=\eps/20$. For $k>20$, let $r_k=2^{k-22}\eps/5$ and  $|x_k|=|x_{k-1}|+r_k$.
Denote by $k_*$ the last $k$ such that  $|x_k| \le 1+3\delta/4.$ Then for all $k\in [0, k_*]$, we have  $|x_k-s| \ge 3r_k/2$ and
  $3r_k/2   \le |x_{k_*}|- 1 \le 3\delta/4$, whence $r_k \le \delta/2$. Note that for $k \le k_{*}$, we have
$$B(x_k,5r_k/4) \subset \Lambda_{\frac{3r_k}{2}+2\eps}(s)\setminus \Gamma\,.
$$
We also have $|x_{k_*}| >1+r_{k_*}$ and $1+3\delta/4 \le |x_{k_*+1}| = |x_{k_*}|+2r_{k_*}$, so
 $$
 |x_{k_*}|-1+r_{k_*}>\bigl(3|x_{k_*}|-3+3r_{k_*}-(|x_{k_*}|-1-r_{k_*})\bigr)/3 =2(|x_{k_*}|+2r_{k_*}-1)/3 \ge  \delta/2 \,.
 $$
  Let $\xi$ be the point on $L(x_0)$ that satisfies $|\xi|=1+3\delta/4$. Observe that $B(\xi, 5\delta/{16}) \subset \Lambda_\delta(s)\setminus \Gamma$.
Applying the preceding lemma to the overlapping balls $\{B(x_k,r_k)\}_{k \le k_*}$  and $B(\xi, \delta/4)$   yields, for some $C_6$, that
\begin{eqnarray}\label{oscnow}
\forall k \le k_{*} \,, \qquad \underset{ B(x_k,r_k)} {\rm osc}~ u  \le C_6 r_k^\frac{p-1-\beta}{p} \,, \qquad
\underset{ B(\xi, \delta/4)} {\rm osc}~ u  \le C_6 \delta^\frac{p-1-\beta}{p} \,.
\end{eqnarray}
 Summing the resulting  (almost geometric) series concludes the proof for the oscillation over the line segment $L(x_0)$. Considering the union of these line segments over all $\{x_0\in \Lambda_{\eps}^*(s) : 1 \le |x_0| \le 1+\delta/4\}$  allows us to bound the oscillation over
$\{ x\in \Lambda_{\eps}^*(s) :  \vert x \vert  \ge 1\}$. A similar argument applies to $\{ x\in \Lambda_{\eps}^*(s) :  \vert x \vert  \le 1\}$.
Since these two sets intersect on the unit sphere, the proof is complete.
\end{proof}

\begin{lemma}\label{l:3}
Fix $\eps$, $\delta,$ $\beta$ as in Definition \ref{defbad}, and let $y$ be a point on the unit sphere   ${\mathbb S}^{d-1}$.  Suppose that
\begin{equation} \label{energy3}
 E(u,\Lambda_{2\delta}(y)) \le M\delta^{d-1},
 \end{equation} for some $M>0$. Then   the set $S_\beta$  of bad anchors satisfies  $|S_\beta \cap Q_\delta(y)| \le C_7 M\delta^\beta (\delta/\eps)^{d-1}$, where $C_7=C_7(d,p)$.
\end{lemma}
The proof is a variant of the classical proof of the Hardy-Littlewood maximal inequality.
\begin{proof}
By the definition of bad anchors, for each $s \in S_\beta \cap Q_\delta(y)$ there is a $ \zeta=\zeta(s) \in[\eps,\delta]$ such that
\begin{eqnarray} \label{eq:bad}
E(u,\Lambda_{\zeta}(s))>  \zeta^{d-1-\beta} \ge  \delta^{-\beta} \zeta^{d-1}.
\end{eqnarray}
The caps  $Q_{\zeta(s)} (s)$ with $s \in S_\beta \cap Q_\delta(y)$   form a Besicovitch covering of $S_\beta \cap Q_\delta(y)$. By the Besicovitch Covering Theorem, we can extract a finite subcover $\{Q_{\zeta_j}(s_j)\}_{j=1}^N$, where $\zeta_j=\zeta(s_j)$ for each $j$,
so that every point on the unit sphere belongs to at most $C_8=C_8(d)$ elements of this subcover.  Therefore,
\begin{eqnarray} \label{besic}
\sum_{j=1}^N {\bf 1}_{\Lambda_{\zeta_j}(s_j)} \le C_8 {\bf 1}_{\Lambda_{2\delta}(y)} \,.
  \end{eqnarray}
 For each $s \in S_\beta \cap Q_\delta(y)$, there is some $j \le N$ such that $s \in Q_{\zeta_j}(s_j)$, whence the cap $Q_{\eps/2}(s) $ is contained in  $Q_{2\zeta_j}(s_j)$.  Since the $\eps/2$ neighborhoods $Q_{\eps/2}(s) $ of anchors are disjoint,
  for some $C_9>0$ we have
\begin{eqnarray} \label{besic2}
C_9 |S_\beta \cap Q_\delta(y)| \eps^{d-1} \le  \mu\Bigl[\bigcup\{Q_{\eps/2}(s) \, : \, s \in S_\beta \cap Q_\delta(y)\}  \Bigr] \le
\mu\Bigl[\bigcup_{j=1}^N Q_{2\zeta_j}(s_j) \Bigr] \le
\sum_{j=1}^N \mu[Q_{2\zeta_j}(s_j)]    \, .
 \end{eqnarray}
 Now  by (\ref{eq:bad}), for each $j$ we have
 \begin{eqnarray} \label{besic3}
 \mu[Q_{2\zeta_j}(s_j)] \le C_{10} \zeta_j^{d-1} \le C_{10} \delta^{\beta} E(u,\Lambda_{\zeta_j}(s_j)) \,.
  \end{eqnarray}
 Thus by \eqref{besic2} and  \eqref{besic3},
\begin{eqnarray} \label{besic4}
 |S_\beta \cap Q_\delta(y)| \eps^{d-1} \le   \frac{C_{10}}{C_9} \delta^{\beta}\sum_{j=1}^N E(u,\Lambda_{\zeta_j}(s_j)) =  C_{11}\delta^{\beta} \int_{{\mathbb R}^d}\Bigl( |\nabla u|^p \sum_{j=1}^N  {\bf 1}_{\Lambda_{\zeta_j}(s_j)} \Bigr)
 \le C_{11} \delta^{\beta}C_8 E(u,\Lambda_{2\delta}(y))    \, ,
 \end{eqnarray}
where we have used (\ref{besic}) in the last step.
Combining (\ref{energy3}) and (\ref{besic4}) concludes the proof.
\end{proof}

\begin{lemma}\label{l:4}
 Let    $\beta:=\frac{p-1}{2p}$. Then there exist $\delta_0>0$ (which may depend on $d,p, \sigma$) and $C_{12},C_{13}>0$ (which  depend on $d,p$) such that for $\delta \in (0,\delta_0)$,  there exists $\eps_0>0$ with the following property. For each $\eps \in (0,\eps_0)$ and     $y\in \mathbb{S}^{d-1}$ such that $u=u^{\eps}$  satisfies
\begin{eqnarray}\label{eq4:1}
E(u,\Lambda_{2\delta}(y)) \le \delta^{-\beta/2} \delta^{d-1} \,,
\end{eqnarray}
there exists some   $A=A(y,\eps)$, such that
\begin{eqnarray}\label{eq4:5}
|u(x)-A| \le C_{12} \delta^\beta, \quad \forall x\in \Lambda_{\delta}(y):  |x|=1\pm\delta \,,
\end{eqnarray}
and
\begin{eqnarray}\label{eq4:2}
E(u,\Lambda_{\delta}(y)) \ge (1- 2\delta^{\beta/3}) E(u_{A},\Lambda_{\delta}(y))\,.
\end{eqnarray}
\end{lemma}
\begin{proof}
 Since $\beta<(p-1)/2$, we have   $(p-1-\beta)/p>(p-1)/2p=\beta$. We will use this repeatedly below.
Also, recall  the following consequence of the definition \eqref{eq:2da3} of $u_A$ and \eqref{eq:musk}, that was already noted in \eqref{decomp-ener2}:
\begin{eqnarray} \label{decomp-ener2.5}
E(u_A, \Lambda_{\delta}(y))\le A^p (1+\eps)^{d-p}E(U, \Lambda_{\delta}(y) )+(1-A)^p \sum_{s \in S \cap Q_{\delta+\eps}(y)} \kappa_s \,,
\end{eqnarray}
where
\begin{eqnarray} \label{decomp-ener3.5}
\kappa_s:=  {\rm cap}_p\left(\alpha \eps K_s, B\left(s,\frac{\eps}{10}\right)\right) \,.
\end{eqnarray}

\medskip

First observe that under the assumption \eqref{eq4:1}, Lemma \ref{l:1} gives that
\begin{eqnarray}\label{eq4:3}
\underset{  \Lambda_{\delta}(y)\cap \partial B(0,1+\delta)} {\rm osc}~ u \le C_1 \delta^\frac{p-1-\beta}{p} \le C_1\delta^{\beta} \,.
\end{eqnarray}

 Similarly,
\begin{eqnarray}\label{eq4:31}
\underset{  \Lambda_{\delta}(y)\cap \partial B(0,1-\delta)} {\rm osc}~ u \le   C_1\delta^{\beta} \,.
\end{eqnarray}

  Hence,
\begin{eqnarray}\label{eq4:4}
u(x)=A_{\pm}+ O(\delta^\beta), \quad \forall x\in \Lambda_{\delta}(y):  |x|=1\pm\delta,
\end{eqnarray}
where $A_+=A_+(y)$ and $A_-=A_-(y)$ are two constants.
In view of Lemma \ref{l:3}, if $\delta_0>0$ is small enough and $\delta \in (0,\delta_0)$, then for sufficiently small $\eps>0$, there exists a good anchor in $Q_{\delta}(y)$, and therefore by Lemma \ref{l:2} and \eqref{eq4:4}, we have $|A_+-A_-|\le C_{13}\delta^\beta.$
Thus
\begin{eqnarray}\label{eq4:51}
|u(x)-A_+| \le C_{14} \delta^\beta, \quad \forall x\in \Lambda_{\delta}(y):  |x|=1\pm\delta \,.
\end{eqnarray}

\medskip
Denote by $Q_{\delta-\eps}^{\sharp}(y)=(S\setminus S_\beta) \cap Q_{\delta-\eps}(y)$   the set of good anchors in $Q_{\delta-\eps}(y).$
 Lemma \ref{l:2} and \eqref{eq4:5} imply that there exists $C_{15}$ such that for each   $s\in Q_{\delta-\eps}^{\sharp}(y)$,
\begin{eqnarray}\label{eq4:6}
u(x) \le A_1:= \min\bigl\{1,A_{+}+C_{15} \delta^\beta \bigr\}  \,, \quad \forall x\in \partial B\left(s,\frac{\eps}{10}\right) \,.
\end{eqnarray}
Define
$$
A:=A_1  \; \mbox{ \rm if } \;  A_+>1/2 \; \mbox{ \rm and } \; A:=\max\{0,A_{+}- C_{14}\delta^\beta \} \; \mbox{ \rm if } \;  A_+ \le 1/2 \,.
$$
We assume that $\delta$ is small enough to ensure that  $C_{15} \delta^\beta <1/4$, so $A=1$ iff $A_1=1$.
Also, note that $A_1 \le A +(C_{14}+C_{15})\delta^\beta$, so $ 1-A_1 \ge (1-A)\bigl(1-2 (C_{14}+C_{15})\delta^\beta\bigr)$ if  $A_+ \le 1/2$. Thus the inequality
\begin{eqnarray}\label{compa}
 (1-A_1)^p \ge (1-\delta^{\beta/3}) (1-A)^p
\end{eqnarray}
 holds for all values of $A_+$, provided that $\delta$ is small enough.

Now that $A$ has been chosen, our next goal will be to prove that if $\delta$ is small enough, then the inequality
\begin{eqnarray}\label{eq4:9}
E\left(u, \Lambda_{\delta}(y)\cap B(0,1+\delta) \right) \ge (1-2\delta^{\beta/3}) (1-A)^p \sum_{s \in S \cap Q_{\delta+\eps}(y)} \kappa_s
\end{eqnarray}
holds for  sufficiently small $\eps$. Note that if $A_1=1$ then also $A=1$,
 so the right-hand side of \eqref{eq4:9} vanishes, and the inequality certainly holds; thus we need only prove \eqref{eq4:9} when $A_1<1$.

  Given $s \in Q_{\delta-\eps}^{\sharp}(y)$, fix a $C^\infty$ cutoff function $\phi_s: B\left(s,\frac{\eps}{10}\right) \to [0,1]$ that
is identically 1 in $B(s,\alpha \eps)$ and vanishes outside $B(s,2\alpha \eps)$. Since $A_1<1$ and \eqref{eq4:6} holds, the function $u_s: B(s, \eps/{10}) \to [0,1]$ given by
$$u_s(x):=\max \Bigl\{0, \frac{u(x)-A_1}{1-A_1}\Bigr\} \,,
$$
   is in the admissible class
 $${\mathcal A}(s+\alpha \eps K_s, \, B(s, \eps/{10}), \, \phi_s)
 $$
  defined in \eqref{eq:bala-1}, by basic properties of $H_0^{1,p}$ (See \cite[Lemmas 1.23 and 1.26]{hkm}).  Thus, by \eqref{eq:bala0},
$$\forall s \in Q_{\delta-\eps}^{\sharp}(y), \qquad  \kappa_s\le E \Bigl(u_s, \, B\left(s,\frac{\eps}{10}\right)\Bigr)\le
(1-A_1)^{-p} E\Bigl(u, B\left(s,\frac{\eps}{10}\right)\Bigr) \,.$$

Therefore,
\begin{eqnarray}\label{eq4:8}
E\left(u,  \Lambda_{\delta}(y)\cap B(0,1+\delta) \right) \ge
\sum_{s\in  Q_{\delta-\eps}^{\sharp}(y)} E\left(u, B\left(s,\frac{\eps}{10}\right)\right)  \ge (1-A_1)^{p}\sum_{s\in  Q_{\delta-\eps}^{\sharp}(y)}    \kappa_s\, .
\end{eqnarray}

Next, we will compare the right-hand-sides of \eqref{eq4:8} and \eqref{eq4:9}.
The equidistribution hypothesis $({\rm  H_1})$ implies that for $\eps$ sufficiently small,
\begin{eqnarray}\label{eq:kid}
|S \cap Q_{\delta}(y)| \ge (1-\delta)\sigma \eps^{1-d} \mu( Q_{\delta}) \ge C_{17} \sigma  (\delta/\eps)^{d-1} \,.
\end{eqnarray}
 Invoking Lemma \ref{l:3} with $M=\delta^{-\beta/2}$ yields that
\begin{eqnarray}\label{eq:kid1}
|S_\beta \cap Q_{\delta}(y)| \le C_7  \delta^{\beta/2} (\delta/\eps)^{d-1}\,.
\end{eqnarray}
The $\mu$-measure of  the shell $Q_{\delta+2\eps}(y) \setminus  Q_{\delta-2\eps}(y)$ is  $O(\eps\delta^{d-2})$, so
 the number of anchors in the smaller shell  $Q_{\delta+\eps}(y) \setminus  Q_{\delta-\eps}(y)$ is
  $$O(\delta/\eps)^{d-2} \le C_7\delta^{\beta/2}(\delta/\eps)^{d-1} $$
(since caps of radius $\eps/2$ around these anchors are pairwise disjoint, and contained in the larger shell.)   Thus for small enough $\eps$,
\begin{eqnarray}\label{eq:kid3}
 \Bigl|\Bigl(S \cap  Q_{\delta+\eps}(y)\Bigr) \setminus  Q_{\delta-\eps}^{\sharp}(y)\Bigr|
 &\le & \Bigl| S_\beta \cap Q_\delta(y)\Bigr| + \Bigl|S \cap \bigl(Q_{\delta+\eps}(y) \setminus  Q_{\delta-\eps}(y)\bigr)\Bigr| \\ \nonumber
  &\le& 2C_7  \delta^{\beta/2} (\delta/\eps)^{d-1} \le  \frac{C_{17}}{2} \sigma\delta^{\beta/3}  (\delta/\eps)^{d-1} \,,
 \end{eqnarray}
 where the rightmost inequality assumes that  $\delta $ is small enough so that $2C_7\delta^{\beta/6} \le \frac{C_{17}}{2} \sigma$.

Lemma \ref{limcaprate} (see also \eqref{eq:separ1}) implies that if $\eps>0$ is small enough, then  for every anchor $s$,
\begin{eqnarray}\label{eq4:85}
   {\rm cap}_p  (\alpha \eps K) \le  \kappa_s   \le 2 {\rm cap}_p  (\alpha \eps K)  \,.
\end{eqnarray}

By comparing  \eqref{eq:kid3} to \eqref{eq:kid}, and using \eqref{eq4:85}, we obtain that
\begin{eqnarray}\label{eq:kid4}
\sum \Bigl\{ \kappa_s \,: \, s\in \Bigl(S \cap  Q_{\delta+\eps}(y)\Bigr) \setminus  Q_{\delta-\eps}^{\sharp}(y)\Bigr\} \le
 \delta^{\beta/3}\sum \Bigl\{ \kappa_s \,: \, s\in S \cap  Q_{\delta+\eps}(y)  \Bigr\} \,,
 \end{eqnarray}
 or equivalently,
 \begin{eqnarray}\label{eq:kid5}
\sum \Bigl\{ \kappa_s \,: \, s\in   Q_{\delta-\eps}^{\sharp}(y)\Bigr\} \ge
 \Bigl(1-\delta^{\beta/3}\Bigr)\sum \Bigl\{ \kappa_s \,: \, s\in  S \cap  Q_{\delta+\eps}(y)   \Bigr\} \,.
 \end{eqnarray}

Combining this inequality with  \eqref{eq4:8} and \eqref{compa}, we have established that \eqref{eq4:9} holds.

\medskip
Recall from \eqref{eq:2} that   $u(x)\to 0$ as $|x|\to\infty$ and recall that $U_R(x)=U(x/R)$ is identically 1 for $ x \in B(0,R)$.
By considering separately the two cases $A_+>1/2$ and $A_+ \le 1/2$, we infer from \eqref{eq4:51} that
 $$ \forall x \in \Lambda_{\delta}(y) \cap \partial B(0,1+\delta), \qquad u(x) \ge  \tilde{A}:=(1-C_{20}\delta^\beta)A \,.
 $$
  Thus,   Lemma \ref{l:holder} implies that
\begin{eqnarray}\label{eq4:10}
E\left(u,\{x\in  \Lambda_{\delta}(y): ~ |x|\ge 1+\delta \} \right) \ge
\tilde{A}^p  E(U_{1+\delta},\Lambda_{\delta}(y)) = (1+\delta)^{d-p}\tilde{A}^p  E(U,  \Lambda_{\delta}(y)) \,.
\end{eqnarray}
 For sufficiently small $\delta$, we have  $(1-C_{20}\delta^\beta)^p \ge 1-2\delta^{\beta/3}$, so if $\epsilon$ is small enough, then
\begin{eqnarray}\label{eq4:11}
E\left(u, \{x\in  \Lambda_{\delta}(y): ~ |x|\ge 1+\delta \}\right)
\ge (1-\delta^{\beta/3}) A^p E(U, \Lambda_{\delta}(y)) \,.
\end{eqnarray}
 Combining \eqref{eq4:9} and \eqref{eq4:11}, we obtain by \eqref{decomp-ener2.5} that \eqref{eq4:2} holds. This completes the proof.

\end{proof}

\section{Asymptotics for $u^\eps$ in the bulk: Proof of Theorem \ref{t:1} } \label{proofthm1}
In this section, we first use the lower bound on the energy of $u$ in cones, obtained in Lemma \ref{l:4}, in conjunction with $u$ minimizing energy globally, to deduce an upper bound for the energy of $u$ in all cones. This will imply that $u$ is close to $u_{A^*}$ on $\partial B(0,1+\delta)$, from which  Theorem \ref{t:1}  will follow easily.
\begin{lemma}\label{l:5}
Let $\beta=\frac{p-1}{2p}$ as in Definition \ref{defbad}.
There exists $ \delta_0>0$ (that may depend on $d,p,\sigma$) and $C_{21}>0$ (that may depend on $d,p$) such that for all
  $\delta \in (0,\delta_0/2)$, points $z\in \mathbb{S}^{d-1}$  and $m>1$, we have
\begin{eqnarray}\label{eq:lem5}
E(u,\Lambda_{\delta}(z))\le  E(u_{A_*},\Lambda_{\delta}(z)) +C_{21}{\delta^{m\beta/3}} \,,
\end{eqnarray}
provided that $\eps$ is sufficiently small. Consequently, there exists $\theta=\theta(d,p)>0$ such that for $\delta \in (0,\delta_0/2)$ and $\eps$   sufficiently small,
\begin{eqnarray}\label{eq:lem5const}
 |u(x)-A_*| \le C_{\#} \delta^{\theta}   \; \;  \mbox{\rm for all} \; \; x\in \R^d \;\; \mbox{\rm such that} \; \; |x|=1\pm \delta \,,
\end{eqnarray}
where $C_{\#}$  does not depend on $\eps,\delta$ (but may depend on $d,p,\tau,\sigma,{\rm cap}_p(K)$).
\end{lemma}
As in the previous section, the constants $C_i$ in the proof only depend on $d,p$.

\begin{proof}
Let $\Omega={\mathbb R}^d \setminus \Gamma$. Observe that for $x \in  {\mathbb R}^d$  and  $y \in {\mathbb S}^{d-1}$, we have
 $$x \in \Lambda_{2 \delta}(y)  	\Leftrightarrow y \in Q_{2\delta}(x/|x|) \,.
 $$
 Therefore, by Fubini,
\begin{eqnarray}\label{fub1}
\underset{{\mathbb S}^{d-1}}\int  E(u, \Lambda_{2 \delta}(y)) \, d\mu(y)= \underset{{\mathbb S}^{d-1}}\int \,\, \underset{\Omega \cap \Lambda_{2 \delta}(y)} \int |\nabla u(x)|^p \, dx \, d\mu(y)=
\underset{\Omega}\int  \underset{Q_{2\delta}(x/|x|)}\int \!\!\!  |\nabla u(x)|^p \, d\mu(y) \, dx  =\mu(Q_{2\delta})  E(u) \,.
\end{eqnarray}
Define
\begin{eqnarray}\label{eq:bady}
 Y_\delta=\Bigl\{y \in {\mathbb S}^{d-1} \,:\, E(u, \Lambda_{2 \delta}(y))>\delta^{-\beta/2} \delta^{d-1} \Bigr\} \,.
\end{eqnarray}
Since $E(u) ={\rm cap}_p(\Gamma) \le {\rm cap}_p(\bar B(0,2))= C_{22}$ and $\mu(Q_{2\delta}) =O(\delta^{d-1})$,  equation \eqref{fub1} implies that
\begin{eqnarray}\label{eq:bady2}
 \mu(Y_\delta)  \le C_{23} \delta^{\beta/2} \,.
\end{eqnarray}
Denote $\chi_m(y)=E(u_{A_*}, \Lambda_{\delta^m}(y))$ and observe that if $\delta^m<\delta_0(d,p,\sigma)$,  then Lemma \ref{l:4} and  \eqref{eq:bady3}
(applied to $\delta^m$ in place of $\delta$)
imply that for sufficiently small $\eps$,
\begin{eqnarray}\label{eq:bady3.5}
E(u, \Lambda_{  \delta^m}(y)) \ge (1- 2\delta^{m\beta/3}) \chi_m(y)  \; \; \mbox{\rm for} \;\;  y \in {\mathbb S}^{d-1} \setminus Y_{\delta^m} \,.
\end{eqnarray}

Next, note that for  $z \in {\mathbb S}^{d-1}$ and   $y \in {\mathbb S}^{d-1} \setminus Q_{\delta+\delta^m}(z)$, we have
$$ x \in \Lambda_{\delta^m}(y) \, \Rightarrow \, \Bigl\{x \in  {\mathbb R}^d \setminus \Lambda_\delta(z) \; \; {\rm and} \;\; y \in Q_{\delta^m}(x/|x|) \Bigr\} \,.
$$
Therefore, by Fubini
\begin{eqnarray}\label{fub2}
\!\!\! \! \underset{{\mathbb S}^{d-1} \setminus Q_{\delta+\delta^m}(z)}  \int \!\! \! \! \! \!\!\! \! E(u, \Lambda_{  \delta^m}(y)) \, d\mu(y)
 = \!\!\! \! \! \underset{{\mathbb S}^{d-1} \setminus Q_{\delta+\delta^m}(z)} \int \!\! \!  \!\int_{\Omega \cap \Lambda_{  \delta^m}(y)} \!\!
  |\nabla u(x)|^p \, dx \, d\mu(y)
  \le \!\!\! \! \!\!\! \! \underset{\Omega \setminus \Lambda_\delta(z)} \int \!\!\! \! \int_{Q_{\delta^m}(x/|x|)} \! \! \! |\nabla u(x)|^p \, d\mu(y) \, dx   \,.
\end{eqnarray}
The right-hand side factors as $\mu(Q_{\delta^m}) E(u, \Omega \setminus \Lambda_\delta(z)) $,  so reversing the order of expressions gives
\begin{eqnarray}\label{fub2.5}
\mu(Q_{\delta^m}) E(u, \Omega \setminus \Lambda_\delta(z)) \ge \!\! \underset{{\mathbb S}^{d-1} \setminus Q_{\delta+\delta^m}(z)}   \int E(u, \Lambda_{  \delta^m}(y)) \, d\mu(y)
 \ge \!\! \underset{{\mathbb S}^{d-1} \setminus [Q_{\delta+\delta^m}(z) \cup   Y_{\delta^m}]} \int (1-2\delta^{m\beta/3}) \chi_m(y) \, d\mu(y) \,,
\end{eqnarray}
where  we have used \eqref{eq:bady3} in the last step. By Fubini,
\begin{eqnarray}\label{fub3}
\int_{{\mathbb S}^{d-1}}  \chi_m(y) \, d\mu(y) =\mu(Q_{\delta^m})  E(u_{A_*}) \,
\end{eqnarray}
and
\begin{eqnarray}\label{fub4}
\int_{ Q_{\delta-\delta^m}(z) }  \chi_m(y) \, d\mu(y)  \le \mu(Q_{\delta^m})  E(u_{A_*}, \Lambda_\delta(z)) \,.
\end{eqnarray}
Write $Y^+=Y_{\delta^m} \cup \Bigl(Q_{\delta+\delta^m}(z) \setminus Q_{\delta-\delta^m}(z)\Bigr)$, so that for small $\delta$,
\begin{eqnarray}\label{fub5}
\mu(Y^+) \le C_{23}\delta^{m\beta/2}+C_{24}\delta^{d-2+m}\le 2C_{23}\delta^{m\beta/2} \,.
\end{eqnarray}
Now \eqref{c13b} implies the bound $\chi_m(y) \le C_{25}\mu(Q_{\delta^m})$ for $\eps$ small enough; together with \eqref{fub5}, it gives
\begin{eqnarray}\label{fub6}
\int_{ Y^+ }  \chi_m(y) \, d\mu(y) \le C_{26} \mu(Q_{\delta^m}) \delta^{m\beta/2}  \,.
\end{eqnarray}
Subtracting \eqref{fub4} and \eqref{fub6} from \eqref{fub3} yields
\begin{eqnarray}\label{fub7}
\underset{{\mathbb S}^{d-1} \setminus [Q_{\delta+\delta^m}(z) \cup   Y_{\delta^m}]} \int   \chi_m(y) \, d\mu(y)
\ge   \mu(Q_{\delta^m}) \Bigl[  E(u_{A_*}, {\mathbb R}^d \setminus \Lambda_\delta(z)) -C_{26}\delta^{m\beta/2} \Bigr] \,,
\end{eqnarray}
whence by \eqref{fub2.5},
\begin{eqnarray}\label{fub8}
  E(u, \Omega  \setminus \Lambda_\delta(z)) \ge
   (1-2\delta^{m\beta/3})  \Bigl[  E(u_{A_*}, {\mathbb R}^d \setminus \Lambda_\delta(z)) -C_{26}\delta^{m\beta/2} \Bigr] \,.
\end{eqnarray}

Thus,
\begin{eqnarray}\label{fub8.5} E(u_{A_*}) \ge  E(u) &=& E(u,  \Lambda_\delta(z)) +E(u, \Omega \setminus \Lambda_\delta(z)) \\ \nonumber
&\ge& E(u,  \Lambda_\delta(z))+
 (1-2\delta^{m\beta/3})  \Bigl[  E(u_{A_*}, \Omega \setminus \Lambda_\delta(z)) - C_{26}\delta^{m\beta/2} \Bigr] \,.
\end{eqnarray}
Rearranging terms and using that there exists $C_{27}$ such that if $\eps$ is small, then  $E(u_{A_*}) \le C_{27}$, we conclude that
 \begin{eqnarray}\label{fub9}
  E(u,  \Lambda_\delta(z)) \le  E(u_{A_*},  \Lambda_\delta(z)) + 2C_{27}\delta^{m\beta/3}\,,
\end{eqnarray}
provided $\delta <\delta_0$ and $\eps<\eps_0(\delta)$ are small enough. Applying this inequality with $2\delta$ in place of $\delta$, the hypothesis of
 Lemma \ref{l:4} is satisfied,
  provided $m$ is chosen to satisfy  $m\beta/3>d$.   Therefore, by \eqref{eq4:5},  we have    that for  sufficiently small $\eps$,
  \begin{eqnarray}\label{fub9.5}
  |u(x)-A| \le C_{12} \delta^\beta  \;\; \mbox{ \rm for all} \;\; x \in \Lambda_\delta(z) \;\; \mbox{\rm such that} \;\; |x|=1\pm \delta \,,
  \end{eqnarray}
 where $A=A(z,\eps)$. By \eqref{eq4:2} and \eqref{fub9},
 \begin{eqnarray}\label{fub9.7}
    (1- 2\delta^{\beta/3}) E(u_{A},\Lambda_{\delta}(z))\, \le \, E(u_{A_*},  \Lambda_\delta(z)) + 2C_{27}\delta^{d}\,,
\end{eqnarray}
 since  $m\beta/3>d$. Now we separate cases.

 \medskip

 \noindent{\bf Case 1.} If $\tau \in (0,\infty)$, then \eqref{fub9.7} and Corollary \ref{cor:en} yield, for sufficiently  small $\eps>0$, that
  \begin{eqnarray}\label{fub10}
   (1-3 \delta^{\beta/3}) \mu(Q_\delta) \varphi_\tau(A)   \le \mu(Q_\delta) \varphi_\tau(A_*)+ 2C_{27}\delta^{d}\,.
\end{eqnarray}
Thus, since $\varphi_\tau(A_*) \le C$, we infer that
$$C_{28}\delta^{\beta/3}  \ge \varphi_\tau(A)-\varphi_\tau(A_*) \ge c_{\#}|A-A_*|^{2p} \,,
$$
where in the right-hand inequality,  $c_{\#}=c_{\#}(d,p,\tau,\sigma,{\rm cap}_p(K))>0$.  Therefore,
$$|A-A_*|\le (C_{28}/c_{\#})^{1/(2p)} \cdot \delta^{\beta/(6p)} \,.
$$
 In conjunction with \eqref{fub9.5}, this yields the final claim of the lemma.

 \medskip

 \noindent{\bf Case 2.} If $\tau =0$, then $A_*=0$, so $E(u_{A_*}) \to 0$ as $\eps \to 0$ by \eqref{limener}. Since
   $$   E(u_{A},  \Lambda_\delta(z))\ge C_{29}A^p \delta^{d-1} $$
by \eqref{eqh:6}, we infer from \eqref{fub9.7} that
  $|A-A_*|=A =O(\delta^{1/p})$ for small $\eps$.

 \medskip

  \noindent{\bf Case 3.} If $\tau =\infty$, then $A_*=1$ and
  \begin{eqnarray}\label{fubnew}
    E(u_{A_*},  \Lambda_\delta(z)) = E((1+\eps)^{ \gamma} U,  \Lambda_\delta(z)) \le C_{30} \delta^{d-1}
\end{eqnarray}
by \eqref{eqh:6}.
  On the other hand, for $A<1$,
   by \eqref{decomp-ener2} and the definition of $\kappa_s$, for  any $\eps>0$ small enough we have
 \begin{eqnarray}\label{fubnew1} \frac{E(u_{A},\Lambda_{\delta}(z))}{(1-A)^p} \ge  \sum_{s \in S \cap Q_{\delta-\eps}(y)} \kappa_s \ge
      C_{31} \sigma (\delta/\eps)^{d-1} (\alpha \eps)^{d-p}{\rm cap}_p(K)\ge
  c_0 \delta^{d-1} (\alpha/\alpha_c)^{d-p}
  \end{eqnarray}
  for some $c_0=c_0(d,p,K,\sigma)>0$. The right-hand side of \eqref{fubnew1} tends to $\infty$ as $\eps \downarrow 0$,   so by \eqref{fub9.7}, we have
  $|A-A_*|=1-A \le \delta$ for small enough $\eps$.
\end{proof}

\begin{proof}[Proof of Theorem \ref{t:1}]
  By Lemma \ref{l:5}, if $\delta<\delta_0(d,p,\sigma)$, then for sufficiently small $\eps$, we have
  $$ \forall x \in \partial B(0,1+\delta) \cup \partial B(0,1-\delta), \qquad  A_*-C_{\#} \delta^\theta  \le u(x) \le A_*+C_{\#} \delta^\theta  \,,
  $$
  where $\theta, C_{\#}$ do not depend on $\eps,\delta$.
The comparison principle (Lemma  \ref{l:comp}) then implies that for sufficiently small $\eps$,
\begin{eqnarray}\label{bulk1}
 \forall x \in \R^d \setminus B(0,1+\delta), \qquad  (A_*-C_{\#} \delta^\theta) U_{1+\delta}(x) \le u(x) \le (A_*+C_{\#} \delta^\theta)U_{1+\delta}(x)  \,.
 \end{eqnarray}
 and
 \begin{eqnarray}\label{bulk2}
 \forall x \in   B(0,1-\delta), \qquad  A_*-C_{\#} \delta^\theta  \le u(x) \le A_*+C_{\#} \delta^\theta    \,.
 \end{eqnarray}
 (For the latter conclusion, the maximum principle suffices.)
 Since $\delta>0$ can be chosen arbitrarily small, this concludes the proof of uniform convergence on compact subsets of $\R^d \setminus \partial B(0,1)$.
\end{proof}

\section{Separation theorem} \label{necksec}

The oscillation and energy estimates we used in the previous section to prove Theorem \ref{t:1} are not sufficient to   establish Theorem \ref{main2}. For this purpose, we will need to bound $u$ closer to the cavities, which is the goal of this section.

Recall that  $\gamma = \frac{d-p}{p-1}$. In this section we prove some estimates that are valid for all parameters $\delta<\min\{1/2,  1/(2\gamma) \}$  and $0<\eps<\delta/10$.
 In particular, we do not assume the asymptotic equidistribution hypothesis $(H_1)$ in  \eqref{eq:4}, and instead of the asymptotic relation $(H_3)$ $\alpha(\eps) \eps^{-1/\gamma} \to \tau$ we just assume that for some $\tau_1 \in (0,\infty)$, we have
$\alpha <\min\{1/80, \tau_1 \eps^{1/\gamma} \} \,.  $

Let $S \subset \partial B(0,1)$ be a set of anchors, with Euclidean distance at least $\eps$ between any two anchors. The next theorem ensures that for $\alpha$ in the critical window and below it, the bumps in the equilibrium potential are separated.
\begin{theorem} \label{neckthm}
For $\zeta>0$, write    $\Omega_\zeta = B(0,1 +\delta) \setminus \bigl[ \cup_{s \in S} \bar{B}(s, \zeta \eps) \bigr]$.
Let $w: \bar{B}(0,1+\delta) \to [0,1]$ be the Perron solution of the  boundary value problem
\begin{eqnarray}
\left\{
\begin{array}{lll}
\Delta_p w = 0 & \mbox{in} &\Omega_{\alpha}, \\
w=0 & \mbox{on} & \partial B(0,1+\delta), \\
w=1 & \mbox{on}   &  \cup_{s\in S} \partial B(s,\alpha \eps) \,.
\end{array}
\right.
\end{eqnarray}
Then for some $C=C(d,p)$, we have
\begin{equation} \label {eq:neckth}
\sup_{z \in \Omega_{1/10}} |w(z)| \le C\tau_1^\gamma \delta \,.
\end{equation}
\end{theorem}

\begin{figure}[h!]%
	\centering \includegraphics[width=4.in]{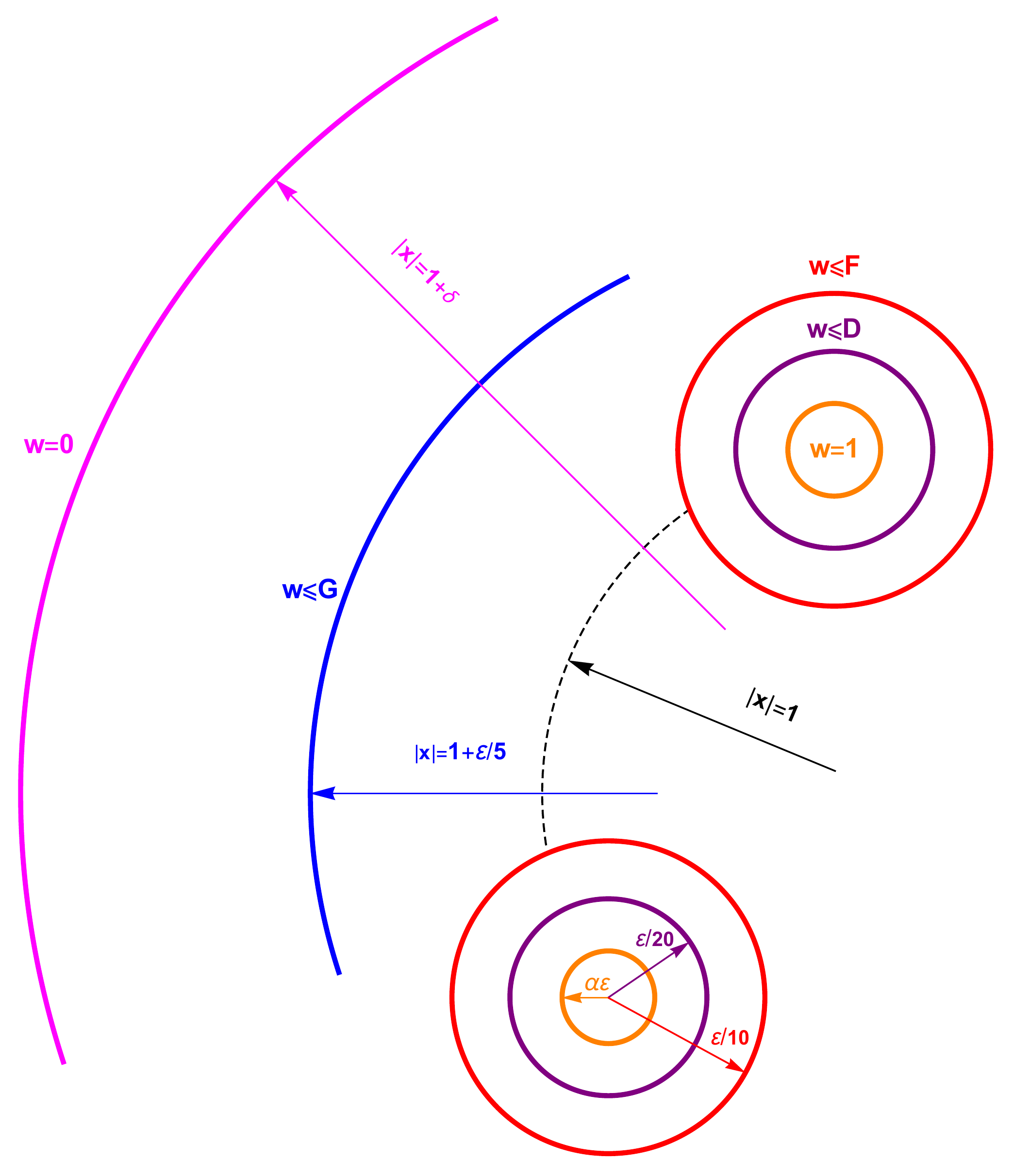}
	\label{fig:neck}
	\caption{The inner balls are centered at anchors on the unit sphere and have radius $\alpha \eps$; on their boundary, $w=1$. Each  of them is surrounded by  two concentric spheres  of radii $\eps/20$ and  $\eps/10$, respectively. All these spheres are contained in  larger spheres of radii $1+ \eps/5$ and $1+\delta$ centered at the origin.}
\end{figure}

The proof of Theorem \ref{neckthm} is based on the lemma below which requires the following notation:
\begin{align}
D & := \sup\Bigl\{ w(x) \mid   x \in \cup_{s \in S} \; \partial B\left (s, \frac{\eps}{20}\right) \Bigr\}\, ,\notag \\
F & := \sup\Bigl\{ w(x) \mid  x \in \cup_{s \in S} \; \partial B\left(s, \frac{\eps}{10}\right) \Bigr\} \, , \notag \\
G & := \sup\Bigl\{ w(x) \mid x \in \partial B\left(0 , 1 + \frac{\eps}{5}\right)\Bigl\} \,. \notag
\end{align}
We note here that by the maximum principle,
\begin{equation} \label{eq:boundviaF}
\sup_{z \in \Omega_{1/10}} w(z) \le F \,.
\end{equation}

\begin{lemma}\label{necklem}
There exist constants $c_1, c_2>0$ and $c_3 \in (0,1)$ that only depend on $d,p$,
such that
\begin{align}
(a) \qquad D & \leq F  + c_1 \tau_1^\gamma  \eps(1- F) \,, \label{D_ineq} \\
(b) \qquad G & \leq (1 - c_2 \eps/\delta)F \label{G_ineq} \,,\\
(c) \qquad F & \leq (1-c_3)D + c_3 G \label{F_ineq} \,.
\end{align}
\end{lemma}
\begin{proof}
Let $0<r<R$ and set
\begin{equation} \label{funda}
h_{r,R}(x)=\frac{|x|^{-\gamma} - R^{-\gamma}}{r^{-\gamma} - R^{-\gamma}}  \; \mbox{ \rm for }\;  r < |x| < R\,.
\end{equation}
Observe that $h_{r,R}$ is $p$-harmonic in $B(0,R)\setminus \bar  B(0,r)$ and takes values $0$ and $1$ on $\partial B(0,R)$ and $\partial B(0,r)$ respectively.

\medskip

(a) ~Given the definition of $F,$  the comparison principle (Lemma \ref{l:comp}) implies that for each $s\in S$, we have
\begin{eqnarray}\label{neckeq1}
w(x)\le F+(1-F)\nu_1(x-s) , \qquad \forall x\in  B\left(s,\frac{\eps}{10}\right) \setminus  B(s,\alpha\eps),
\end{eqnarray}
where $\nu_1 (x)=h_{r,R}(x)$ with $r=\alpha \eps$,~$R={\eps}/{10}.$

Since
\begin{eqnarray} \label{neckeq2}
|x|=\eps/20 \; \Longrightarrow \; \nu_1(x) = \frac{(\eps/20)^{-\gamma} - (\eps/10)^{-\gamma} }{(\alpha \eps)^{-\gamma} - (\eps/10)^{-\gamma}} \leq
\frac{(\eps/20)^{-\gamma}} {(\alpha \eps)^{-\gamma} }
\le c_1 \tau_1^\gamma \eps \,,
\end{eqnarray}
we infer from \eqref{neckeq1} that
\begin{eqnarray}
  w(x) \leq F + (1-F) c_1\tau_1^\gamma  \eps    \qquad \forall x\in \partial B\left(s,\frac{\eps}{20}\right) \,.
\end{eqnarray}
This gives \eqref{D_ineq}.

\medskip

(b) ~ Observe that by \eqref{eq:boundviaF} we have  $w \le F$ on $\partial B(0,1+\eps/10).$ Hence,  by the comparison principle, for all $s\in S$  we have
\begin{eqnarray}\label{neckeq3}
w(x)\le F \nu_2 (x) , \qquad \forall x\in B\left(0, 1+\delta \right) \setminus B\left (0,1+\frac{\eps}{10}\right)\,,
\end{eqnarray}
where $\nu_2(x)=h_{r,R}(x)$ with $r=1+\eps/10$ and $R=1+\delta.$

Since the derivative of $t \mapsto t^{-\gamma}$ is bounded above and below by positive constants for $t \in [1,2]$, the mean value theorem gives
  \begin{equation} \label{vlocal}
  1 -  \nu_2\left(1+\frac{\eps}{5}\right)  = \frac{(1+\eps/10)^{-\gamma} - (1 + \eps/5)^{-\gamma}}{(1+\eps/10)^{-\gamma} - (1 + \delta)^{-\gamma}} \geq c_2 \frac{\eps}{\delta}\,.
\end{equation}
  Combining \eqref{neckeq3} and  \eqref{vlocal}, we have that
\begin{eqnarray}
w(x) \le \left ( 1-c_2\frac{\eps}{\delta}\right) F, \qquad  \forall x\in\partial B\left(0,1+\frac{\eps}{5}\right).
\end{eqnarray}
  This gives \eqref{G_ineq}.

\medskip

(c) ~ Denote by $\xi$  the standard basis vector $(1,0,\ldots,0)$ in   ${\mathbb R}^d,$ and consider the open cored half-ball
\begin{equation}
H:=\{x=(x_1,\ldots,x_d) \in B(0,8)\setminus \bar{B}(4\xi,1) \; : x_1>0 \}.
\end{equation}

\begin{figure}[h!]%
\centering \includegraphics[width=4in]{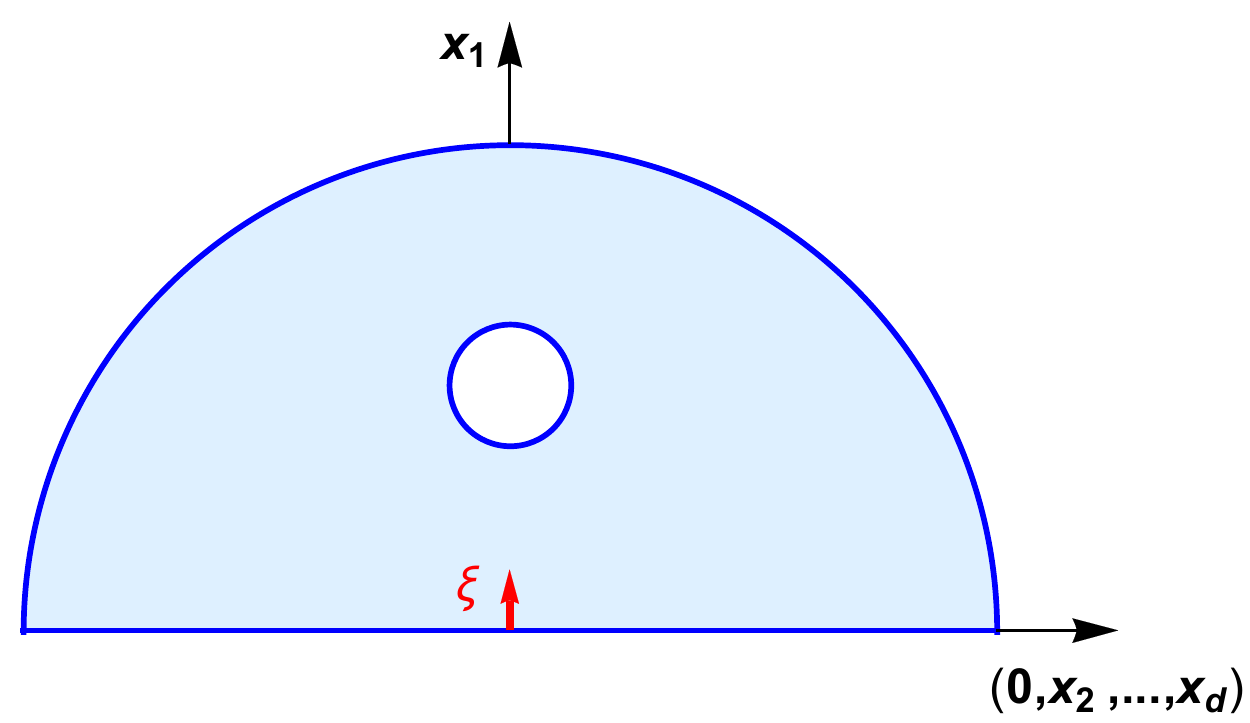}
	\label{fig:cored}
	\caption{The   cored half-ball
$H$ obtained by removing a unit ball from the middle of a half-ball of radius 8.}
\end{figure}

Let $\psi$ be the  Perron solution  of $\Delta_p \psi =0$ in $H$, with boundary values  $\psi =0$ on the flat part $\{ x \in \partial H \,: \,   x_1=0\}$ of $\partial H$, and $\psi = 1$ on the curved part $\{ x \in    \partial B(0,8) \,: x_1>0\} \cup \partial B(4\xi,1)$. (By  Proposition 9.31 in \cite{hkm}, the upper and lower Perron solutions coincide since the boundary conditions are lower semi-continuous and $H$ is regular.) By the strong maximum principle (see \cite{Lindq}, Corollary \ 2.22), we have
$\psi(x) <1$ for all $x \in H.$  By the continuity of $\psi$ in $H$, for some $c_3>0$, we have
\begin{equation} \label{psimax}
 \sup \{\psi(x) \,: \, x \in \partial B(4\xi,2)\} =1-c_3 \,.
\end{equation}
Fix $s \in S$ and let $s_*=(1+\eps/5)s$. The cored half-ball
\begin{equation}
H_s:=\{x \in B(s_*, 2\eps/5)\setminus \bar{B}(s,\eps/20) \; : \langle x,s \rangle<1+\eps/5 \} \,
\end{equation}
is a rotated and translated copy of $H$, scaled by $\eps/20$. The function $w$ is continuous on $\bar{H_s}$, and $p$-harmonic in $H_s$, with boundary values  $w \le G$ on the flat part of $\partial H_s$ and  $w \le D$ on the curved part. Therefore, by the comparison principle and \eqref{psimax}, we have
\begin{equation} \label{psimax2}
 \sup \{w(x) \,: \, x \in \partial B(s,\eps/10)\} \le G+ (1-c_3)(D-G)=(1-c_3)D+c_3 G \,.
\end{equation}
This proves (\ref{F_ineq}).
\end{proof}

\begin{proof}[{\bf Proof of Theorem \ref{neckthm}}]
Combining \eqref{D_ineq}-\eqref{F_ineq}, we have
\begin{eqnarray}
& F  \leq (1-c_3) \bigl[ F + c_1 \tau_1^\gamma \eps (1-F)\bigr] + c_3 \bigl[1 - c_2 \eps/\delta \bigr]F.
\end{eqnarray}
 This  gives
 \begin{eqnarray}
 0  \leq (1-c_3)c_1\tau_1^\gamma \eps(1-F) - c_3 c_2(\eps/\delta) F\,,
\end{eqnarray}
which in turn implies that
\begin{eqnarray}
 c_3 c_2 (\eps/\delta)F \leq   c_1 \tau_1^\gamma \eps\,,
 \end{eqnarray}
 whence
 \begin{eqnarray}
& F \leq \left(\frac{c_1 }{c_3 c_2} \right) \tau_1^\gamma  \delta \,.
\end{eqnarray}
Thus there exists a constant $C_{32}>0$ such that
$F \leq C_{32}\tau_1^\gamma \delta$. By \eqref{eq:boundviaF}, this completes the proof.
\end{proof}

\section{Asymptotics for $u^\eps$ near the unit sphere: Proof of Theorem \ref{main2} } \label{proofmain}
We will need Clarkson's inequalities, in the slightly more general form given in \cite{Boas} for functions $f,g$ taking values in   ${\mathbb R}^d$:
\begin{eqnarray} \label{clark1}
 p \ge 2 \Rightarrow \quad \left\| \frac{f + g}{2} \right\|_{ p}^p + \left\| \frac{f - g}{2} \right\|_{p}^p \le
 \frac{1}{2} \left( \| f \|_{p}^p + \| g \|_{ p}^p \right) \,,
\end{eqnarray}
and
\begin{eqnarray} \label{clark2}
  1<p \le 2 \Rightarrow  \quad\left\| \frac{f + g}{2} \right\|_{ p}^q + \left\| \frac{f - g}{2} \right\|_{ p}^q \le
 \left( \frac{1}{2} \| f \|_{ p}^p +\frac{1}{2} \| g \|_{ p}^p \right)^\frac{q}{p} \,,
 \end{eqnarray}
 where $\frac1{p} + \frac1{q} = 1$.

\begin{proof}[Proof of Theorem \ref{main2}]
 Recall that $\beta=(p-1)/(2p)$. As before, we let $C_i$ denote constants that depend only on $d,p$. Pick a small $\delta>0$.
 We invoke Lemma  \ref{l:5}  with $m$ chosen to satisfy $m\beta/3 \ge d$, and infer from \eqref{c13b}  that
$$
\forall y \in {\mathbb S}^{d-1}, \quad E(u, \Lambda_{2\delta}(y)) \le E(u_{A_*},\Lambda_{2\delta}(y))+2^d C_{21} \delta^d \le C_{33}\delta^{d-1} \,,
$$
provided $\eps$ is small enough.
Thus the hypothesis of Lemma  \ref{l:4}  holds for every $y \in {\mathbb S}^{d-1}$, so that lemma yields
\begin{eqnarray}\label{eq4:2new}
E(u,\Lambda_{\delta}(y)) \ge (1- 2\delta^{\beta/3}) E(u_{A},\Lambda_{\delta}(y))\,.
\end{eqnarray}
for some $A$ that may depend on all parameters of the problem, including $\delta$ and $\eps$.
In conjunction with \eqref{eq:bady3}, this yields that for small enough $\delta$,
\begin{eqnarray}\label{eq4:2new1}
E(u,\Lambda_{\delta}(y)) \ge (1- 3\delta^{\beta/3}) E(u_{A_*},\Lambda_{\delta}(y))\,.
\end{eqnarray}
By averaging this inequality  over $y \in {\mathbb S}^{d-1}$,   Fubini's Theorem  yields (as in the proof of \eqref{fub1}) that for sufficiently small $ \eps $,
\begin{eqnarray}
E(u) \mu(Q_\delta) \ge (1-3\delta^{\beta/3}) E(u_{A_*})\mu(Q_\delta)   \,.
\end{eqnarray}
We deduce that for $\eps$ small enough,
\begin{eqnarray}\label{soclose}
0 \le   E(u_{A_*}) -E(u) \le C\delta^{\beta/3} \,.
\end{eqnarray}
 Let $\phi$ be the cutoff function  defined in Remark \ref{rem:adm}. Then $u, u_{A_*} $  are both in the convex admissible class
${\mathcal A}(\Gamma, {\mathbb R}^d, \phi)$
 defined in \eqref{eq:bala-1}, and $u$ minimizes the energy in this class, so the inequality
 $$E\Bigl(\frac{u+u_{A_*}}{2} \Bigr) \ge E(u)$$
holds. In conjunction  with \eqref{soclose} and   Clarkson's inequalities, this implies that
$$\|\grad u -\grad u_{A_*}\|_p=O(\delta^c)$$
for some constant $c>0$ and all sufficiently small $\eps$. Since $\delta$ can be arbitrarily small, this proves the first statement of the theorem.

It only remains to prove the  $L^\infty$ convergence.

Recall that we have already established in \eqref{bulk1} that for small $\delta$ and (given $\delta$) sufficiently small $\eps$, we have
\begin{eqnarray} \label{estab}
\forall x \in \R^d \setminus B(0,1+\delta), \qquad  (A_*-C_{\#} \delta^\theta) U_{1+\delta}(x) \le u(x) \le (A_*+C_{\#} \delta^\theta)U_{1+\delta}(x)  \,.
\end{eqnarray}
for some $0< \theta<1$.

If $\tau=\infty$,  then $A^*=1$ and the first inequality in \eqref{estab}, together with the minimum principle,
 imply that $u \ge  1-C_{\#} \delta^\theta$ in $B(0,1+\delta)$, completing the proof in this case. Thus, we may assume that $\tau<\infty$.

Next, recall the continuous $p$-harmonic function
$w$ from Theorem \ref{neckthm} defined in the regular domain
$$\Omega_{\alpha} = B(0,1 +\delta) \setminus \bigl[ \cup_{s \in S} \bar{B}\bigl(s, \alpha\eps  \bigr) \bigr]
$$
considered there.  Theorem \ref{neckthm}  implies that $w \le C_{34}(2\tau)^{\gamma} \delta$ in the closure of the domain
$$\Omega_{1/10} = B(0,1 +\delta) \setminus \bigl[ \cup_{s \in S} \bar{B}\bigl(s,  \eps/10  \bigr) \bigr] \,.
$$
 The inequality $u(x)-A_*-C_{\#}\delta ^\theta \le w$ holds on
 $\partial \Omega_{\alpha}$, so it must holds in $\Omega_{\alpha}$ as well, by
the comparison principle (Lemma \ref{l:comp}). In particular, for $\delta<\delta_0$ small enough (that may depend on $d,p,\sigma,\tau$), and $\epsilon$ small
enough given $\delta$, we have
$$\forall x \in \partial B(s,\eps/10), \quad u \le A_*+C_{\#}\delta^\theta +C_{34}(2\tau)^{\gamma} \delta  \le A_*+2C_{\#}\delta ^\theta  \,. $$
On the other hand, $u \ge  A_*-C_{\#} \delta^\theta$ in $B(0,1+\delta)$ by \eqref{estab}   and the minimum principle (or by Lemma \ref{l:comp}.)
Thus
\begin{eqnarray}\label{so4}
 |u-u_{A_*}| \le 2C_{\#} \delta^\theta  \quad \text{in} \quad  \Omega_{1/10} \,.
\end{eqnarray}

Finally, fix an anchor $s \in S$ and let   $D_s:=B(s,\eps/10) \setminus (s+\alpha \eps K_s)$.  Observe that $u_{A^*}$, restricted to $D_s$,
is the Perron solution of the $p$-Laplace equation in $D_s$, with the boundary condition $f_s$ which is identically 1 on $s+\alpha \eps K_s$ and equals $A_*$ on
$\partial B(s,\eps/10)$.
The inequality
\begin{eqnarray}\label{soclose3}
u_{A_*}(x) -2C_{\#} \delta^\theta  \le u(x) \le u_{A_*}(x)+2C_{\#} \delta^\theta
\end{eqnarray}
  holds for $x \in \partial D_s$.

Let $\psi$ be a superharmonic function in  the upper class  $U_f^\Omega$ (see Definition \ref{def:upper}) corresponding to the boundary conditions
$$ f \equiv 1 \; \text{on} \; \Gamma, \quad f(\infty)=0$$
 in the definition \eqref{eq:2} of  $u$ as a Perron solution in $\Omega={\mathbb R}^d \setminus \Gamma$. Then $\psi+2C_{\#} \delta^\theta $, restricted to $D_s$, is in the
upper class  $U_{f_s}^{D_s}$  by \eqref{soclose3}, so $\psi+2C_{\#} \delta^\theta  \ge u_{A_*}$ in $D_s$. Taking an infimum over all such $\psi \in U_f^\Omega$,
we infer that
 \begin{eqnarray} \label{close4}
 u+2C_{\#} \delta^\theta  \ge u_{A_*} \quad \text{in} \;  D_s \,.
 \end{eqnarray}

 Similarly, let $\psi_1$ be a superharmonic function in  the upper class  $U_{-f}^\Omega$. Then $\psi_1+2C_{\#} \delta^\theta $, restricted to $D_s$, is in the
upper class  $U_{-f_s}^{D_s}$ by \eqref{soclose3}, so $\psi_1+2C_{\#} \delta^\theta  \ge -u_{A_*}$ in $D_s$. Taking an infimum over all such $\psi_1 \in U_{-f}^\Omega$,
we infer that
 \begin{eqnarray} \label{close5}
 -u+2C_{\#} \delta^\theta  \ge -u_{A_*} \quad \text{in} \;  D_s \,.
 \end{eqnarray}
   Combining this with \eqref{close4} concludes the proof, since $\delta$ can be arbitrarily small.
\end{proof}

\section{Concluding remarks.}
\noindent{\bf 1.} One method to obtain an asymptotically equidistributed set of anchors is to choose $S(\eps)$ as any
$\epsilon$-separated set on the sphere of maximal cardinality. The asymptotic equidistribution
\begin{eqnarray}\label{eq:4maak}
\frac{1}{|S(\eps)|}
 \sum_{s\in S(\eps)} \delta_s \overset{\ast} \rightharpoonup   \mu \: ~ \mbox{as} ~ \eps\downarrow 0,
\end{eqnarray}
 then follows from a classical argument of Maak, as presented, e.g., in
\cite[Chapter 12]{Don}, while the existence of the limit  $\sigma=\lim_{\eps \downarrow 0} \eps^{d-1} |S(\eps)|$ is established in \cite{Ham}.
\medskip

\noindent{\bf 2.} The main results of this paper directly extend  to problem \eqref{eq:2} considered
in a ball $B(0,R)\subset \mathbb{R}^d$  with $R>1.$  Indeed, let $1<p<d,$   $R>1$ and let $w=w^{\eps}$ be the Perron solution of the following problem
\begin{eqnarray}\label{eq:2conc}
\left\{
\begin{array}{lll}
\Delta_p w=0 &\mbox{in} & B(0,R) \setminus  \Gamma,\\
w=1 & \mbox{on} & \Gamma,\\
w = 0 & \mbox{on} & \partial B(0,R).
\end{array}
\right.
\end{eqnarray}
Define an ansatz
\begin{eqnarray}\label{eq:2da3cor}
w^{\eps}_{A}(x) := A W_{\eps,R}(x) +(1-A) \sum_{s\in S}V_{\frac{1}{10\alpha}}^s  \left(\frac{x-s}{\alpha \eps}\right)\,,
\end{eqnarray}
where $W_{\eps,R}$ is the $p$-equilibrium potential of the   ball $\bar{B}(0,1+\eps)$  relative to the ball $B(0,R)$,
  given by
\begin{eqnarray}\label{eq:solUcor}
W_{\eps,R}(x):=\left\{
\begin{array}{ll}
1& 0\le |x| \le 1+\eps \,,\\
\frac{|x|^{-\gamma}-R^{-\gamma}} {(1+\eps)^{-\gamma}-R^{-\gamma}}& 1+\eps<|x|\le R \,,
\end{array}
\right.
\end{eqnarray}
and $V_\rho^s$ are as in \eqref{eq:V}.

The following corollary holds:
\begin{cor}
Suppose that hypotheses ${\rm (H_1),(H_2),(H_3)}$ hold.
  Then, as $ \eps\to 0$,
\begin{eqnarray} \label{eq:thm1Bcor}
w^\eps(x)\to \left\{
\begin{array}{lll}
0& \mbox{if} & \tau=0 \,,\\
A_R W_{0,R}(x) & \mbox{if} &\  \tau\in(0,\infty) \,,\\
W_{0,R}(x) &  \mbox{if} & \tau= \infty \,
\end{array}
\right.
\end{eqnarray}
uniformly on compact subsets of  $\bar B(0,R) \setminus \mathbb{S}^{d-1},$
where for $\tau \in [0,\infty)$,
\begin{eqnarray}\label{eq:B*}
A_R=A_R(\tau)= \frac{\left(\sigma \tau^{d-p} {\rm cap}_p(K)\right)^\frac{1}{p-1}}{\left(\sigma \tau^{d-p} {\rm cap}_p(K)\right)^\frac{1}{p-1}+\left({\rm cap}_p(\bar B(0,1),B(0,R))\right)^\frac{1}{p-1}} \,.
\end{eqnarray}
Furthermore, as $ \eps\to 0$,
\begin{eqnarray} \label{eq:thm1Bcor2}
{\rm cap}_p(\Gamma_\eps, B(0,R))\to \left\{
\begin{array}{lll}
0& \mbox{if} & \tau=0 \,,\\
A_R^p {\rm cap}_p (\bar B(0,1),B(0,R))+(1-A_R)^p  {\rm cap}_p (K) \sigma \tau^{d-p} & \mbox{if} &\  \tau\in(0,\infty) \,,\\
{\rm cap}_p(\bar B(0,1),B(0,R))  &  \mbox{if} & \tau= \infty \, .
\end{array}
\right.
\end{eqnarray}
Moreover,
\begin{eqnarray}
\Vert \nabla w-\nabla w^{\eps}_{A_R}\Vert_{{L^p(\bar B(0,R))}}\to 0,
\end{eqnarray}
and
\begin{eqnarray}
\Vert w-w_{A_R}\Vert_{L^{\infty}(\bar B(0,R))}\to 0 \,.
\end{eqnarray}
\end{cor}
The proof of this corollary is a line by line adaptation of the arguments in the preceding sections.

\bigskip

\noindent {\bf Acknowledgments.} The work of PVG was supported in a part by the Simons Foundation grant 317882  and US-Israel BSF grant 2020005.
The work of FN was partially supported by the NSF grant DMS-1900008 and US-Israel BSF grant 2020019.

\end{document}